\documentclass[11pt]{amsart}

\usepackage{enumerate}
\usepackage{bbm}           
\usepackage{bm}
\usepackage{eso-pic,graphicx}
\usepackage{tikz}
\usepackage{ulem}

\usepackage{amsmath,amsthm,amscd,amssymb,amsfonts, amsbsy}
\usepackage{latexsym}
\usepackage{txfonts}
\usepackage{exscale}
\usepackage{mathrsfs}

\usepackage[colorlinks=true, pdfstartview=FitV, linkcolor=blue, citecolor=blue, urlcolor=blue, pagebackref]{hyperref}

\calclayout
\allowdisplaybreaks

\theoremstyle{plain}
\newtheorem{theorem}[equation]{Theorem}
\newtheorem{lemma}[equation]{Lemma}
\newtheorem{corollary}[equation]{Corollary}
\newtheorem{proposition}[equation]{Proposition}

\theoremstyle{definition}
\newtheorem{definition}[equation]{Definition}

\theoremstyle{remark}
\newtheorem{remark}[equation]{Remark}

\numberwithin{equation}{section}

\renewcommand{\Re}{{\rm Re}\,}

\newcommand{\BMO}{\mathcal{BMO}}

\def\R{{\mathbb R}}

\def\({\left(}
\def\){\right)}
\def\[{\left[}
\def\]{\right]}
\def\<{\left<}
\def\>{\right>}

\def\R{\mathbb R}

\def\B{\mathcal B}

\def\essup{\text{ess\,sup}}

\DeclareSymbolFont{bbold}{U}{bbold}{m}{n}
\DeclareSymbolFontAlphabet{\mathbbn}{bbold}
\allowdisplaybreaks

\newcommand{\bw}{\textbf{w}}
\newcommand{\bp}{\textbf{P}}

\begin{document}
\allowdisplaybreaks

\title[Commutators with BMO functions]
{Boundedness results for commutators with BMO functions via weighted estimates: A comprehensive approach}


\author[\'A. B\'enyi]{\'Arp\'ad B\'enyi}

\address{\'Arp\'ad B\'enyi
\\
Department of Mathematics
\\
516 High St
\\
Western Washington University
\\
\break Bellingham, WA 98225, USA} \email{arpad.benyi@wwu.edu}

\author[J.M. Martell]{Jos\'e Mar{\'\i}a Martell}

\address{Jos\'e Mar{\'\i}a Martell
\\
Instituto de Ciencias Matem\'aticas CSIC-UAM-UC3M-UCM
\\
Consejo Superior de Investigaciones Cient{\'\i}ficas
\\
C/ Nicol\'as Cabrera, 13-15
\\
E-28049 Madrid, Spain} \email{chema.martell@icmat.es}

\author[K. Moen]{Kabe Moen}

\address{Kabe Moen
\\
Department of Mathematics
\\
University of Alabama
\\
P.O. Box 870350
\\
Tuscaloosa, AL 35487, USA} \email{kabe.moen@ua.edu}

\author[E. Stachura]{Eric Stachura}

\address{Eric Stachura
\\
Department of Mathematics and Statistics
\\
Haverford College
\\
370 Lancaster Ave
\\
Haverford, PA 19041, USA} \email{estachura@haverford.edu}

\author[R. Torres]{Rodolfo H. Torres}
\address{Rodolfo H. Torres
\\
Department of Mathematics
\\
University of Kansas
\\
1460 Jayhawk Blvd
\\
Lawrence, Kansas 66045-7523, USA} \email{torres@ku.edu}

\thanks{B\'enyi and Moen were partially supported by grants from the Simons Foundation No. 246024 and  No. 160427, respectively. Martell acknowledges financial support from the Spanish Ministry of Economy and Competitiveness, through the ``Severo Ochoa" Programme for Centres of Excellence in R\&D (SEV-2015- 0554). He also acknowledges that the research leading to these results has received funding from the European Research Council under the European Union's Seventh Framework Programme (FP7/2007-2013)/ ERC agreement no. 615112 HAPDEGMT. Torres was supported in part by NSF grant DMS 1069015.}



\date{\today}

\begin{abstract}
We present a unified method to obtain unweighted and weighted estimates of linear and multilinear commutators with $BMO$ functions, that is amenable to a plethora of operators and functional settings. Our approach elaborates on a commonly used Cauchy integral trick, recovering many known results but yielding also  numerous new ones. In particular, we solve a problem about the boundedness of the commutators of the bilinear Hilbert transform with functions in $BMO$.
\end{abstract}

\maketitle

\tableofcontents

\section{Introduction}
The purpose of this work is to present a panoramic point of view on the topic of commutators of linear, multilinear, and linearizable operators with pointwise multiplication by functions which belong to the John-Nirenberg space of Bounded Mean Oscillation ($BMO$). We will work in the weighted $L^p$ setting, where the weights will satisfy appropriate Muckenhoupt conditions. We will give a unified method for the weighted estimates of commutators with $BMO$ functions, first in the linear setting, and then in the multilinear setting.

Our approach is based on an idea going back to the pioneer work of Coifman, Rochberg and Weiss \cite{CRW1976}, who first studied the commutators of Calder\'on-Zygmund operators with $BMO$ functions. They used what is now sometimes referred to as the {\it Cauchy integral trick} to obtain boundedness properties of commutators from weighted estimates for the operator itself.  The idea is to represent the commutator as the derivative at the origin of an analytic family of operators obtained by conjugating the given operator with complex exponentials. Using the Cauchy integral to represent such derivative, estimates for the commutator follow by Minkowski's inequality and the fact that exponentials of $BMO$ functions are essentially Muckenhoupt weights.

This remarkably simple but elegant approach has been already generalized to other situations. For example, it was substantially extended in the work of Alvarez, Bagby, Kurtz, and P\'erez  \cite{ABKP}, where particular classes of pairs of weights were used, and iterated commutators and some vector valued ones were studied as well. Their work provides a lot of background and motivation for ours. Other applications have been used in the context of spaces of homogeneous type, for example in the work of Hofmann, Mitrea, and Taylor \cite{HMT2009}.  It is quite remarkable that such a general and {\it soft} approach to commutator estimates produces, in some cases, sharp weighted estimates in terms of the character of the weights. This was more recently shown by Chung, Pereyra, and P\'erez \cite{CPP2012}. We will further extend the Cauchy integral trick so it can be applied to many new situations not covered by previous extensions of the method. In particular, we will be able to treat operators which map a Lebesgue space into another one with a different exponent, multilinear operators, and operators which satisfies weighted estimates but only for limited classes of weights.  We will also be able to apply our approach in very general measure theoretic and geometric contexts and multiparameter settings.  We will recover many known results and establish several new ones as well.

The main new idea is to replace $BMO$ by its ${\rm exp}\,L$ Orlicz representation denoted by $\BMO$,  which allows for a better quantification of some estimates and more flexibility in applications. Of course, in the classical Euclidean setting $BMO=\BMO$, but this may not be the case in more general contexts.

Nowadays, there is a myriad of papers on the topic of weighted norm inequalities for  commutators, both linear and multilinear. Although most of the works are based on the same ideas, typically, the arguments have to be reworked every time since the operators or the settings are different. Independently, for these settings and operators it is common to find a well established weighted norm theory.  Our results say that, with the exception of some end-point results, the estimates for the commutators follow, almost automatically, from the weighted theory. Thus, in most cases, the ad-hoc arguments for the commutators can be skipped, provided the weighted norm inequalities have been proved. Our main result for linear operators is as follows (we will provide detailed definitions in the next section).

\begin{theorem}\label{thm:main}
Let $T$ be a linear operator. Fix $1< p, q < \infty$, $\theta > 0$, and $1 < s< \infty$. Suppose that  there exists  an increasing function $\phi: [1, \infty) \longrightarrow  [0, \infty)$  such that for all weights $w$ so that $w^\theta \in A_s$, and all $f \in L^p (w^p)$,
\begin{equation}\label{eqn:main:hypintro}
\|Tf\|_{L^q (w^q)}
\leq
\phi \left( \big[w^\theta\big]_{A_s} \right) \|f\|_{L^p (w^p)}.
\end{equation}
Then, for each $b\in \BMO$ and for each weight $w$ and $1<\eta<\infty$ such that $w^{\theta\,\eta}\in A_s$,
\begin{equation}\label{eqn:main:k=1intro}
\big\|[T, b]f \big\|_{L^q (w^q)}
\leq  C(\phi,w,\theta,\eta,s)\,
\|b\|_{\BMO}\, \|f\|_{L^p (w^p)};
\end{equation}
and, for every $k\ge 2$,
\begin{equation}\label{eqn:main:kintro}
\big\|[T, b]_k f \big\|_{L^q (w^q)}
\leq  C(\phi,w,\theta,\eta,s,k)\,
\|b\|_{\BMO}^k\, \|f\|_{L^p (w^p)}.
\end{equation}
\end{theorem}

The apparently cumbersome form of the statement of the above theorem allows us to apply it to several families of operators. For example, for  classical Calder\'on-Zygmund operators one simply takes  $p=q=\theta=s=2$. However, other situations require one to work with a power of a weight in a given class. Consider, for example, the operator $L=-{\rm div }(A\nabla )$, where the matrix $A$ is uniformly elliptic, bounded and complex. One can study  the Riesz transform, functional calculus, and square function associated with this operator. All these satisfy weighted norm inequalities in $L^p(w^p)$ (or $L^p(u)$ after writing $u=w^p$) for some range of $p$ and for some class of weights which can be written as $w^\theta\in A_s$ with $\theta$, $s$ depending on the index $p$ and the range of the un-weighted estimates. Another familiar situation is the one given by the classical fractional integrals or Riesz potentials. There are natural weighted norm inequalities for some $p<q$ and for weights $w\in A_{p,q}$ or, equivalently, $w^q\in A_{1+q/p'}$ (that is, $w^\theta\in A_s$ for some $\theta>1$ and $s>1$). The same kind of behavior occurs for fractional powers of the operators $L$ above.  A look at the existing literature (we will give precise reference in later sections) reveals that the commutators of all the operators in the previous families with point-wise multiplication by $ BMO$ functions satisfy weighted norm inequalities, but the arguments need to be adapted to each operator in question. Our result says that once the weighted norm theory has been developed, the estimates for commutators follow at once.

Another important feature that we would like to emphasize is that the way the hypotheses are written  using  the $\BMO$ space is amenable to several immediate generalizations, see Theorem \ref{thm:main:B}. For instance, we could replace the basis of cubes in $\BMO$ and the $A_p$ conditions by other families of sets such as dyadic cubes, rectangles, and so on, as well as replace the underlying space $\mathbb R^n$ by a space of homogeneous type and more. In general situations, the John-Nirenberg inequality or the self-improvement of the $A_p$ classes -- that is, the fact that $w\in A_p$ implies  $w^{1+\epsilon}\in A_p$ -- may fail (see \cite[p. 29]{Chema2010}). However, as we will see in the proof of our results, our formulation does not require any of these deep properties. The $\BMO$ space we use here is defined using the exponential norm so that {\it it carries the John-Nirenberg inequality on its DNA}. Hence, $\BMO$ is in general smaller than the usual $BMO$ space defined using the $L^1$-averages. Nonetheless, in addition to the classical Euclidean setting, $\BMO$ does coincide with $BMO$ also in some multiparameter setting; see Theorem \ref{thm:bmo-BMO} below. Note that the conclusion in Theorem~\ref{thm:main}, is written for weights $w$ such that $w^{\theta\,\eta}\in A_s$ for some $\eta$, which, in the Euclidean setting with the Lebesgue measure and with the Muckenhoupt conditions written in terms of cubes or balls, happens to be true for every $w^\theta\in A_s$ by the fact that $A_p$-weights satisfy reverse H\"older inequalities. This, however, may not be true for more general measures or geometries. When the reverse H\"older inequality does hold, we  can give very precise estimates on the constants appearing in Theorem \ref{thm:main}; see Theorem \ref{thm:main:II}.

We obtain in Theorem \ref{thm:Comm-vs-Weights} a partial converse to Theorem \ref{thm:main}, and multilinear versions in Theorems \ref{thm:multilinear:main:I} and \ref{thm:multilinear:main:II}. As a corollary we prove the first quantitative results for commutators of multilinear Calder\'on-Zygmund operators, obtaining estimates in terms of the character of the weights; see Theorem \ref{thm:multilinear:CZ}. A second important corollary is the proof of the boundedness of the commutators of the bilinear Hilbert transform with $BMO$ functions. Although the first results about boundedness of the bilinear Hilbert transform by Lacey and Thiele \cite{LT}, \cite{LT2}  have been know for some 20 years now, the boundedness of its commutators was a problem that eluded the efforts of several researchers. Our methods provide the result including also some weighted estimates; see Theorems~\ref{thm:BHT:commutator}  and Corollary~\ref{corolo:BHT-new}.   Moreover, combining this result with recent extrapolation ones by Cruz-Uribe and Martell \cite{CM-extrapol}, we prove in Corollary~\ref{corol:BHT:commutator} the boundedness of the commutators in the same range of Lebesgue exponents known for the boundedness of the bilinear Hilbert transform itself.

We collect most of the basic definitions and properties of weights needed in the next section; the expert reader can easily skip this part. The rest of the article is organized as follows: the main results for linear or linearizable operators are presented in Section \ref{linear}, the general theorems in the multi-linear case are given in Section \ref{multilinear}, and a multitude of applications are outlined in Section \ref{applications}. Finally, Appendix \ref{appendix} contains some further results for multilinear operators.  
\bigskip

{\bf Acknowledgement}.
Part of this work was carried out while the second and two last named authors were together at the Instituto de Ciencias Matem\'aticas (ICMAT), Madrid, Spain, as participants of the Research Term on Real Harmonic Analysis and Applications to PDE in 2013. Since then the authors have presented this work at many venues and added several new applications of the methods initially developed. Most recently the last author presented this work  at the Mathematical Sciences Research Institute (MSRI), Berkeley, California in February 2017, during the Harmonic Analysis program.  The authors would like
to thank the ICMAT, the MSRI,  and the organizers of the programs mentioned for
providing such a great atmosphere to carry out mathematical research.

\section{Preliminaries}
Let $|E|$ denote the Lebesgue measure of a measurable set $E\subseteq \mathbb{R}^n$. Throughout this paper, we use the notation
$$
f_E=\fint_E f \,dx = \frac{1}{|E|}\int_E f(x) \,dx
$$
with the understanding that the term is zero if the set $E$ in question has zero or infinite Lebesgue measure. We also use $Q$ to denote cubes with sides parallel to the coordinate axes.

\subsection{Muckenhoupt weights}\label{subsection:Muckenhoupt}
By a weight, we mean a non-negative, locally integrable function. We now recall the definition and several properties of Muckenhoupt's $A_p$ classes of weights. For $1 < p < \infty$,  a weight $w$ is  said to belong to $A_p$  provided
$$
[w]_{A_p}:=\sup_{Q\subset \mathbb{R}^n}\left( \fint_Q w \,dx \right) \left( \fint_Q w^{1-p'} \,dx \right)^{p-1}<\infty;
$$
as usual, $p'$ denotes the conjugate exponent of $p$. For $p=1$, $w$ is in $ A_1$ if
$$
[w]_{A_1}:=\sup_{Q\subset \R^n}\left(\fint_Q w \,dx\right)\,\|w^{-1}\|_{L^\infty(Q)}<\infty,
$$
or, equivalently, $
[w]_{A_1}
=
\|Mw/w\|_{L^\infty(\mathbb{R}^n)}.
$
Here and elsewhere $M$ denotes the Hardy Littlewood maximal function,
$$
M f(x) = \sup_{Q\ni x} \fint_Q |f(y)|\, dy.
$$
We set
$$
A_\infty = \bigcup_{p \geq 1} A_p.
$$
We say that a weight $w$ belongs to the reverse H\"older class $RH_q$, $1 < q < \infty$, if
$$
[w]_{RH_q}:=\sup_{Q\subset\mathbb{R}^n} \left( \fint_Q w^q \,dx \right)^{1/q} \Big(\fint_Q w \,dx\Big)^{-1}<\infty.
$$
When $q = \infty$, we say that $w \in RH_\infty$ if
$$
[w]_{RH_\infty}:=\sup_{Q\subset\mathbb{R}^n} \|w^{-1}\|_{L^\infty(Q)}\Big(\fint_Q w \,dx\Big)^{-1}<\infty.
$$
In the case that $q =1$ we define $RH_1$ to be all $A_\infty$ weights.

We list a few properties that will be important for us later on. The proofs are standard, and can be found for instance in \cite{RDF1985}. For $(g)$, see \cite{JN1991}.
\begin{proposition}\label{prop:weights}
\
\begin{list}{$(\theenumi)$}{\usecounter{enumi}\leftmargin=1cm \labelwidth=1cm \itemsep=0.2cm \topsep=0cm \renewcommand{\theenumi}{\alph{enumi}}}

\item $A_1\subset A_p\subset A_q$ for $1\le p\le q<\infty$ with $[w]_{A_1}\ge [w]_{A_p}\ge [w]_{A_q}$.

\item $RH_{\infty}\subset RH_q\subset RH_p$ for $1<p\le q\le \infty$ with $[w]_{RH_\infty}\ge [w]_{RH_q}\ge [w]_{RH_p}$.

\item If $w\in A_p$, $1<p<\infty$, then there exists $1<q<p$ such
that $w\in A_q$.

\item If $w\in RH_q$, $1<q<\infty$, then there exists $q<p<\infty$ such
that $w\in RH_p$.

\item $\displaystyle A_\infty=\bigcup_{1\le p<\infty} A_p=\bigcup_{1<q\le
\infty} RH_q $

\item If $1<p<\infty$, $w\in A_p$ if and only if $w^{1-p'}\in
A_{p'}$.

\item If $1\le q\le \infty$ and $1\le s<\infty$, then $\displaystyle
w\in A_q \cap RH_s$ if and only if $ w^{s}\in A_{s\,(q-1)+1}$.
\end{list}
\end{proposition}

As we will work in the weighted setting, we need the notion of weighted $L^p$ spaces: $L^p (w)=L^p(\mathbb{R}^n,wdx)$ denotes the collection of measurable functions $f$ on $\mathbb{R}^n$ such that
$$
\|f\|_{L^p (w)}:= \left( \int_{\mathbb{R}^n} |f(x)|^p w(x) \,dx \right)^{1/p} <  \infty.
$$

\subsection{The space of functions of bounded mean oscillation}
We recall here the definition of the John-Nirenberg space of functions of bounded mean oscillation, $BMO(\mathbb R^n)$, and indicate how it identifies with an appropriate Orlicz space. We say that a locally integrable function $f\in BMO(\mathbb R^n)$ if
$$
\|f\|_{BMO} := \sup_Q \fint_Q |f - f_Q| \,dx < + \infty;
$$
the supremum is taken over the collection of all cubes $Q\subset\mathbb{R}^n$. We will simply write $BMO$ instead of $BMO(\mathbb R^n)$ when it is understood that the underlying space on which it is defined is the Euclidean space.

Now, given a cube $Q\subset\mathbb{R}^n$, the localized and normalized $\exp L$-norm is defined as
$$
\|f\|_{\exp L,Q}
:=
\inf \left\{ \lambda > 0 : \fint_Q \bigg(e^{\frac{|f(x)|}{\lambda}}-1\bigg) \,dx \leq 1 \right\}
$$$$=
\inf \left\{ \lambda > 0 : \fint_Q e^{\frac{|f(x)|}{\lambda}} \,dx \leq 2 \right\}.
$$
Note that,
$$
 \fint_Q \frac{|f-f_Q|}{ \|f-f_Q\|_{{\rm exp}\,L,Q}} \, dx
\leq
\fint_Q \left(e^{\frac{|f-f_Q| }  { \|f-f_Q\|_{{\rm exp}\,L,Q} }} -1\right)\, dx \leq 1
$$
or
$$
 \fint_Q |f-f_Q| \, dx \leq  \|f-f_Q\|_{{\rm exp}\,L,Q},
$$
and hence
\begin{equation} \label{easy}
\|f\|_{BMO}=\sup_{Q} \fint_Q |f-f_Q| \, dx \leq \sup_{Q} \|f-f_Q\|_{{\rm exp}\,L,Q}.
\end{equation}

On the other hand, the John-Nirenberg inequality says that there exists a dimensional constant $C_n$ such that for every $f\in BMO$ we have
\begin{equation}\label{hard}
\|f\|_{BMO}
\le
\sup_Q \|f-f_Q\|_{{\rm exp}\,L,Q}
\le
C_n
\|f\|_{ BMO}.
\end{equation}

The estimate \eqref{hard} captures the exponential integrability of a $BMO$ function, a fact that is going to be crucial for our results about commutators and weights. 
Therefore, it is convenient to re-normalize $BMO$ and use the exponential norms as this will make the computations and the constants in our arguments cleaner.  Thus, we can define an a priori new Orlicz-type space $\BMO(\mathbb R^n)$ (or simply $\BMO$) via the norm
$$
\|b\|_{ \BMO} := \sup_Q \|b - b_Q\|_{\exp L, Q}.
$$
Clearly, \eqref{hard} implies that in the classical Euclidean setting $(BMO,\|\cdot\|_{ BMO})$ and $(\BMO,\|\cdot\|_{ \BMO})$ are equivalent quasi-normed spaces; nevertheless, the appearance of $\BMO$ in the statements of our results is to emphasize that $\|\cdot\|_{ \BMO}$ is used.

\subsection{Commutators}
We recall the notion of commutators with linear operators. Suppose $T$ is a linear operator. Define, whenever it makes sense, the (first order) commutator of $T$ with a measurable function $b$ by
\begin{align} \label{commutator}
T_b^1 f(x)
=
[T, b] f(x)
=
T(bf) (x) - b(x) Tf(x)
=
T\big(\big(b(\cdot)-b(x)\big)f(\cdot)\big)(x).
\end{align}
One can also define the higher order commutators of $T$ with a measurable function $b$ by the recursive formula $T_b^k =[T_b^{k-1},b]$ for every $k\ge 2$. Instead of $T_b^k$, the notation $[T, b]_k$ is sometimes used in the literature. It is immediate to see that in such a case,
$$
T_b^k f(x)
=
T\Big(\big(b(\cdot)-b(x)\big)^k f(\cdot)\Big)(x),
\qquad
k\ge 0,
$$
where it is understood that $T_b^0=T$.

The previous definition can be extended to linearizable  operators. We say that a sublinear operator $T$ is linearizable if there exists a Banach space $\mathbb{B}$ and a $\mathbb{B}$-valued linear operator $\mathcal{T}$ such that $T f(x)=\|\mathcal{T}f(x)\|_{\mathbb{B}}$. In this way we set, for all $k\ge 0$,
$$
T_b^k f(x)
:=
\|\mathcal{T}_b^k f(x)\|_{\mathbb{B}}
=
\Big\|
\mathcal{T}\Big(\big(b(\cdot)-b(x)\big)^k f(\cdot)\Big)(x),
\Big\|_{\mathbb{B}}
=
T\Big(\big(b(\cdot)-b(x)\big)^k f(\cdot)\Big)(x).
$$

\section{The general theorem}\label{linear}

\subsection{Commutators in the classical setting I}\label{section:Comm-class-I}
The following result is a generalized version of the classical theorem of Coifman-Rochberg-Weiss \cite{CRW1976} which will cover many of the commutators results in the literature. It is a more precise formulation of Theorem \ref{thm:main} stated in the introduction. Moreover, the formulation can be further sharpened using stronger properties of the Muckenhoupt weights, but  we will do this in Section \ref{section:Comm-class-II}. For now, we prefer to state our result in a way that can be generalized to some other situations where the associated Muckenhoupt weights may fail to verify such properties. Recall that we have the identification of spaces $BMO=\BMO$. The operators we consider here act on functions defined on $\mathbb R^n$.

\begin{theorem} \label{thm:main:I}
Let $T$ be either a linear or a linearizable operator. Fix $1 \leq p, q < \infty$, $\theta > 0$, and $1 < s< \infty$. Suppose that  there exists  an increasing function $\phi: [1, \infty) \longrightarrow  [0, \infty)$  such that for each $f \in L^p (w^p)$ we have
\begin{equation}\label{eqn:main:hyp}
\|Tf\|_{L^q (w^q)}
\leq
\phi \left( \big[w^\theta\big]_{A_s} \right) \|f\|_{L^p (w^p)},
\qquad
\forall\, w^\theta \in A_s.
\end{equation}
Then, for each $b \in \BMO$, for each weight $w$ and $1<\eta<\infty$ such that $w^{\theta\,\eta}\in A_s$ we have
\begin{equation}\label{eqn:main:k=1}
\big\|[T,b]f \big\|_{L^q (w^q)}
\leq
\frac{\eta'\,\theta}{\min \left\{1, s-1 \right\}}\,
\phi
\left( 4^{ \frac{\min\{1,s-1\}}{\eta'}}\, \big[w^{\theta\,\eta}\big]_{A_s}^{1/\eta} \right)\,
\|b\|_{\BMO}\, \|f\|_{L^p (w^p)};
\end{equation}
and, for every $k\ge 2$,
\begin{equation}\label{eqn:main:k}
\big\|T_b^k f \big\|_{L^q (w^q)}
\leq
k!\left(\frac{\eta'\,\theta}{\min \left\{1, s-1 \right\}}\right)^k\,
\phi
\left( 4^{ \frac{\min\{1,s-1\}}{\eta'}}\, \big[w^{\theta\,\eta}\big]_{A_s}^{1/\eta} \right)\,
\|b\|_{\BMO}^k\, \|f\|_{L^p (w^p)}.
\end{equation}
\end{theorem}

\subsection{Proof of Theorem \ref{thm:main:I}}
It is well-known that $\BMO$ functions and $A_\infty$ weights are related. Indeed one can show that if $w\in A_\infty$, then $\log w\in \BMO$. The converse is ``almost'' true: if $b\in \BMO$, then $e^{\lambda\,b}\in A_\infty$ provided $|\lambda|$ is sufficiently small. A first quantitative version of this result was given by Pereyra in \cite[Lemma 3]{pereyra}. We present a similar quantification, computing precisely the parameters involved in terms of $\|b\|_{\BMO}$.

\begin{lemma}\label{lemma:BMO->Ap}
Let $b \in  \BMO$. Then $e^{\lambda\,b}\in A_{1+|\lambda|\,\|b\|_{ \BMO}}$ for every $|\lambda|\le \|b\|_{\BMO}^{-1}$; moreover,
\begin{equation}\label{eqn:BMO->Ap:1}
\big[e^{\lambda\,b}\big]_{A_{1+|\lambda|\,\|b\|_{ \BMO}}} \leq 4^{|\lambda|\,\|b\|_{ \BMO}}.
\end{equation}
In particular, for each $1\le p<\infty$ and $\lambda\in\mathbb{R}$ verifying
\begin{equation}\label{eqn:BMO->Ap:2}
|\lambda| \leq \frac{\min\left\{1, p-1 \right\}}{\|b\|_{ \BMO}},
\end{equation}
we have
$e^{\lambda\,b} \in A_p$ and
\begin{equation}\label{eqn:BMO->Ap:3}
\big[e^{\lambda\,b}\big]_{A_p} \leq 4^{|\lambda|\,\|b\|_{ \BMO}}.
\end{equation}
\end{lemma}

\begin{proof}
By homogeneity considerations we may assume that $\|b\|_{ \BMO} = 1$. Then, for every cube $Q$, we have
\begin{equation}\label{eqn:b-BMO-1}
\fint_Q e^{|b - b_Q|} \,dx \leq 2.
\end{equation}
Now, given $|\lambda|\le 1$, let $q=1+|\lambda|$. We have to prove that $e^{\lambda\,b}\in A_{q}$ with $\big[e^{\lambda\,b}\big]_{A_q}\le 4^{|\lambda|}$.

The case $\lambda=0$ is trivial. Assume otherwise that $0<|\lambda| \le 1$. By Jensen's inequality and \eqref{eqn:b-BMO-1} we have
\begin{align} \label{eqn-AP-BMO1}
\fint_Q e^{|\lambda|\, |b- b_Q|} \,dx \leq \left(\fint_Q e^{|b-b_Q|} \,dx\right)^{|\lambda|}\le 2^{|\lambda|}.
\end{align}
On the other hand, since $q'-1=|\lambda|^{-1}$, \eqref{eqn:b-BMO-1} implies that
\begin{align}\label{eqn-AP-BMO2}
\fint_Q e^{|\lambda|\, |b-b_Q|\, (q' -1)} \,dx
=
\fint_Q e^{|b-b_Q|} \,dx \leq 2.
\end{align}
These estimates and the fact that $q\le 2$ imply that
\begin{align*}
\left(\fint_Q e^{\lambda\,b} \,dx\right)\, &
\left(\fint_Q e^{\lambda\,b\,(1-q')} \,dx\right)^{q-1}\\
& =
\left(\fint_Q e^{\lambda\,(b-b_Q)} \,dx\right)\,
\left(\fint_Q e^{\lambda\,(b-b_Q)(1-q')} \,dx\right)^{q-1}
\\
 & \le
\left(\fint_Q e^{|\lambda|\,|b-b_Q|} \,dx\right)\,
\left(\fint_Q e^{|\lambda|\,|b-b_Q|(q'-1)} \,dx\right)^{q-1}
\\
&\le
2^{|\lambda|}\,2^{q-1}
=
4^{|\lambda|}.
\end{align*}
Taking the supremum over all the cubes $Q\subset\mathbb{R}^n$ we conclude that $e^{\lambda\,b}\in A_q$ with $\big[e^{\lambda\,b}\big]_{A_q}\le 4^{|\lambda|}$.

Let us now obtain \eqref{eqn:BMO->Ap:3}. The case $p=1$ is trivial since $\lambda=0$. Suppose next that $1<p\le 2$; in this case, \eqref{eqn:BMO->Ap:2} implies that $1+|\lambda|\,\|b\|_{ \BMO}\le p$ and $|\lambda|\le  \|b\|_{ \BMO}^{-1}$. Then, we can apply \eqref{eqn:BMO->Ap:1} and Proposition \ref{prop:weights} $(a)$ to obtain \eqref{eqn:BMO->Ap:3}. Finally, if $2<p<\infty$ , we observe that \eqref{eqn:BMO->Ap:2} implies that $|\lambda|\le  \|b\|_{ \BMO}^{-1}$ and therefore \eqref{eqn:BMO->Ap:1} holds. Notice that $1+|\lambda|\,\|b\|_{ \BMO}\le 2<p$  and we can get the desired estimate as in the previous case.
\end{proof}

\begin{lemma}\label{lemma:BMO->Ap:prod}
Fix $1 < r < \infty$ and $1 < \eta < \infty$. If $w^\eta\in A_r$ and $b\in \BMO$, then $w\,e^{\lambda\,b}\in A_r$ for every $\lambda\in\mathbb{R}$ verifying
\begin{equation}\label{eqn:BMO->Ap:prod:1}
|\lambda| \leq \frac{\min \left\{ 1, r-1 \right\}}{\eta'\, \|b\|_{ \BMO}}.
\end{equation}
Furthermore,
\begin{equation}\label{eqn:BMO->Ap:prod:2}
\big[w\,e^{\lambda\,b}\big]_{A_r} \leq [w^\eta]_{A_r}^\frac1{\eta}\, 4^{|\lambda|\,\|b\|_{\BMO}}.
\end{equation}
\end{lemma}

\begin{proof}
Let $U = w\, e^{\lambda\,b}$. H\"older's inequality yields
\begin{align*}
&\left(\fint_Q U(x) \,dx \right)\,
\left(\fint_Q U^{1-r'} \,dx\right)^{r-1}
=
\left(\fint_Q w\, e^{\lambda\,b} \,dx \right)\,
\left(\fint_Q w^{1-r'}\,e^{\lambda\,b\,(1-r')} \,dx\right)^{r-1}
\\
&\qquad\le
\left( \fint_Q w^\eta \,dx \right)^{\frac1\eta}\, \left( \fint_Q e^{\lambda\,\eta'\,b} \,dx \right)^{\frac1{\eta'}} \,
\left( \fint_Q  \big(w^\eta\big)^{1-r'}\,dx \right)^{\frac{r-1}{\eta}}\, \left( \fint_Q e^{(1-r')\,\lambda\,\eta'\,b} \,dx \right)^{\frac{r-1}{\eta'}}
\\
&\qquad\le
\big[w^\eta\big]_{A_r}^{\frac1\eta}\,\big[e^{\lambda\,\eta'\,b}\big]_{A_r}^{\frac1{\eta'}}
\le
\big[w^\eta\big]_{A_r}^{\frac1\eta}\,4^{|\lambda|\,\|b\|_{ \BMO}},
\end{align*}
where we have used that $|\lambda|\,\eta'\le \min\{1,r-1\}/\|b\|_{ \BMO}$ which, by Lemma \ref{lemma:BMO->Ap},  implies that $e^{\lambda\,\eta'\,b}\in A_r$ with $\big[e^{\lambda\,\eta'\,b}\big]_{A_r}\le 4^{|\lambda|\,\eta'\,\|b\|_{ \BMO}}$. This readily leads to the desired estimate.
\end{proof}

\begin{proof}[Proof of Theorem \ref{thm:main:I}]
We start with the case of a linear operator $T$. Fix $b\in \BMO$ and assume that $\|b\|_{ \BMO} = 1$. Fix $1<\eta<\infty$ and $w$ such that  $w^{\eta\,\theta}\in A_s$. Set
$$
\delta := \frac{\min \left\{ 1, s-1 \right\}}{\theta\, \eta'}.
$$
Write, whenever it makes sense, $\Psi(z)f(x): = e^{-z\,b(x)} T (e^{z\,b}\,f)(x)$.
Fix $z\in\mathbb{C}$ such that $|z|\le\delta$ and set $W = w\,e^{-\Re(z)\,b}$. By Lemma \ref{lemma:BMO->Ap:prod} we have that
$$
\big[W^\theta\big]_{A_s}
=
\big[w^\theta\,e^{-\Re(z)\,b\,\theta} \big]_{A_s}
\le
\big[w^{\theta\,\eta}\big]_{A_s}^{\frac1{\eta}}\,4^{\theta\,|\Re(z)|}
\le
\big[w^{\theta\,\eta}\big]_{A_s}^{\frac1{\eta}}\,4^{\theta\,\delta}
$$
since our choice of $\delta$ gives that
$$
\theta\,|\Re(z)|
\le
\theta\,\delta
=
\frac{\min\{1,s-1\}}{\eta'}.
$$
Thus, we can apply \eqref{eqn:main:hyp} and obtain
$$
\|\Psi(z) f \|_{L^q (w^q)}
=
\|T (e^{z\,b}\, f)\|_{L^q (W^q)}
\le
\phi \left( \big[W^\theta\big]_{A_s} \right) \|e^{z\,b}f\|_{L^p (W^p)},
$$
hence
\begin{equation}\label{est:Psi-z}
\|\Psi(z) f \|_{L^q (w^q)}\le
\phi \left( \big[w^{\theta\,\eta}\big]_{A_s}^{\frac1{\eta}}\,4^{\theta\,\delta}\right) \|f\|_{L^p (w^p)}.
\end{equation}
On the other hand, we notice that
$$
\Psi '(z)f(x) = -b(x)\,e^{-z\,b(x)}\, T (e^{z\,b}\,f)(x) + e^{-z\,b(x)}\,T( b\, e^{z\,b}\, f)(x).
$$
and by the Cauchy Integral Formula, we have that
$$
[T,b] f(x)
=
\Psi'(0)f(x)
= \frac{1}{2\pi i} \int_{|z|=\delta} \frac{\Psi(z)f(x)}{z^2}\, dz.
$$
See \cite{ABKP} for a rigorous proof of this commonly used formula.

Next, we use Minkowski's inequality and \eqref{est:Psi-z} to conclude that
\begin{align*}
\|[T,b] f\|_{L^q(w^q)}
&\leq \frac{1}{2\pi \delta^2} \int_{|z| = \delta} \|\Psi (z) f\|_{L^q (w^q)} |dz|
\\
&\le
\frac{1}{2\pi \delta^2} \int_{|z| = \delta} \phi \left( \big[w^{\theta\,\eta}\big]_{A_s}^{\frac1{\eta}}\,4^{\theta\,\delta}\right) \|f\|_{L^p (w^p)} |dz|
\\
&\le
\frac{1}{\delta}\,\phi \left( \big[w^{\theta\,\eta}\big]_{A_s}^{\frac1{\eta}}\,4^{\theta\,\delta}\right) \|f\|_{L^p (w^p)}.
\end{align*}
Plugging in the value of the parameter $\delta$ in the previous estimate we obtain \eqref{eqn:main:k=1}.

For $T$ linearizable we proceed in the same way but replacing $T$ by the $\mathbb{B}$-valued linear operator $\mathcal{T}$. In such a case we need to work with the Bochner space $L^q_{\mathbb{B}}(w^q)$ whose associated norm is $\|f\|_{L^q_{\mathbb{B}}(w^q)}=\big\|\|f\|_{\mathbb{B}}\big\|_{L^q(w^q)}$. This is still a Banach space and then in the previous argument it is legitimate to use Minkowski's inequality. The rest of the argument extends \textit{mutatis mutandis} and the details are left to the interested reader.

We now consider the $k$th order commutator. As before, it suffices to assume that $T$ is linear as the linearizable case follows using the same scheme. The proof is very similar after observing that the Cauchy Integral Formula gives
$$
T_b^k f(x)
=
\Psi^{(k)}(0)f(x)
= \frac{k!}{2\pi i} \int_{|z|=\delta} \frac{\Psi(z)f(x)}{z^{k+1}}\, dz.
$$
This and the previous argument readily leads to the desired estimate.
\end{proof}

\subsection{Commutators in general settings}\label{section:Comm-general}
As mentioned before, our main result can be easily extended to more general settings. We can replace the Euclidean setting and the Muckenhoupt classes associated with the basis of cubes by any measure space (we do not even need to have a distance) and basis formed by  collection of sets. The commutator result is abstract and needs a few ingredients. In several interesting applications, however, one needs to get back to friendlier settings.

Let $(\mathscr{X},\Sigma,\mu)$ be a measure space where $\Sigma$ is a $\sigma$-algebra and $\mu$ is a non-negative measure.  A basis $\mathscr{B}\subset \Sigma$ is a collection of sets such that each $B\in\mathscr{B}$ is $\mu$-measurable and $0<\mu(B)<\infty$.
We can define the $A_p$ classes associated with $\mathscr{B}$. A weight $w$ is a non-negative function such that $\int_B w\,d\mu<\infty$ for every $B \in \mathscr{B}$. We say that a weight $w \in A_{p, \mathscr{B}}$, $1 < p < \infty$, if
$$
[w]_{A_{p,\mathscr{B}}}:=\sup_{B \in \mathscr{B}}\left( \fint_B w \,d\mu \right) \left( \fint_B w^{1- p'} \,d\mu \right)^{p-1} < \infty.
$$
Notice that for our results we do not define the ``problematic'' class $A_{1,\mathscr{B}}$. As before, for a given weight $w$, we write $L^p(w)=L^p(w\,d\mu)$. The following properties are trivial and we omit the proof.
\begin{proposition}\label{prop:weights:B}
\
\begin{list}{$(\theenumi)$}{\usecounter{enumi}\leftmargin=1cm \labelwidth=1cm \itemsep=0.2cm \topsep=0cm \renewcommand{\theenumi}{\alph{enumi}}}

\item $A_{p,\mathscr{B}}\subset A_{q,\mathscr{B}}$ for $1<p\le q<\infty$ with $[w]_{A_{p,\mathscr{B}}}\ge [w]_{A_{q,\mathscr{B}}}$.

\item If $1<p<\infty$, $w\in A_{p,\mathscr{B}}$ if and only if $w^{1-p'}\in
A_{p',\mathscr{B}}$.

\end{list}
\end{proposition}

We can also define a $\BMO$ space associated with the family $\mathscr{B}$. We say that a measurable function  $f \in  \BMO_{\mathscr{B}}(\mathscr{X})$ or simply $f \in  \BMO_{\mathscr{B}}$ (when the underlying space $\mathscr{X}$ is clear from the context) provided $\int_B |f|\,d\mu<\infty$ for every $B \in \mathscr{B}$ and
$$
\|f\|_{\BMO_{\mathscr{B}}}
:=\sup_{B \in \mathscr{B}} \|f-f_B\|_{{\rm exp}\,L,B}<\infty.
$$
Clearly, when $\mathscr{X}=\mathbb R^n$, $\mathscr{B}$ denotes the basis of cubes and $d\mu=dx$, we have
$\BMO_{\mathscr{B}}=\BMO=BMO$.

The notation $f_B$ is, as before, for the $\mu$-average of $f$ on $B$, while the norm $\|f\|_{{\rm exp}\,L,B}$ is defined as
$$
\inf \left\{ \lambda > 0 : \fint_B \bigg(e^{\frac{|f|}{\lambda}}-1\bigg) \,d\mu \leq 1 \right\}
=
\inf \left\{ \lambda > 0 : \fint_B e^{\frac{|f|}{\lambda}} \,d\mu \leq 2 \right\}.
$$
This definition of $\BMO_{\mathscr{B}}$ carries the John-Nirenberg inequality as part of the definition and therefore it is what we need for our purposes. One could also define a $BMO_{\mathscr{B}}$ space using $L^1$ averages. In general, we have
$$\BMO_{\mathscr{B}}\subset BMO_{\mathscr{B}},$$
and, unless some further properties are assumed, one cannot expect these two spaces to coincide.

We have the following commutator result.

\begin{theorem} \label{thm:main:B}
Let $T$ be either a linear or a linearizable operator. Fix $1 \leq p, q < \infty$, $\theta > 0$, and $1 < s< \infty$. Suppose that  there exists  an increasing function $\phi: [1, \infty) \longrightarrow  [0, \infty)$ such that for each $f \in L^p (w^p)$ we have
\begin{equation}\label{eqn:main:hyp:B}
\|Tf\|_{L^q (w^q)}
\leq
\phi \left( \big[w^\theta\big]_{A_{s, \mathscr{B}}} \right) \|f\|_{L^p (w^p)},
\qquad
\forall\, w^\theta \in A_{s,\mathscr{B}}.
\end{equation}
Then, for each $b \in  \BMO_{\mathscr{B}}$, for each weight $w$ and $1<\eta<\infty$ such that $w^{\theta\,\eta}\in A_{s,\mathscr{B}}$ we have
\begin{equation}\label{eqn:main:k=1:B}
\big\|[T,b]f \big\|_{L^q (w^q)}
\leq
\frac{\eta'\,\theta}{\min \left\{ 1, s-1 \right\}}\,
\phi
\left( 4^{ \frac{\min\{1,s-1\}}{\eta'}}\, \big[w^{\theta\,\eta}\big]_{A_{s,\mathscr{B}}}^{1/\eta} \right)\,
\|b\|_{\BMO_\mathscr{B}}\, \|f\|_{L^p (w^p)};
\end{equation}
and, for every $k\ge 2$,
\begin{equation}\label{eqn:main:k:B}
\big\|T_b^k f \big\|_{L^q (w^q)}
\leq
k!\left(\frac{\eta'\,\theta}{\min \left\{ 1, s-1 \right\}}\right)^k\,
\phi
\left( 4^{ \frac{\min\{1,s-1\}}{\eta'}}\, \big[w^{\theta\,\eta}\big]_{A_{s,\mathscr{B}}}^{1/\eta} \right)\,
\|b\|_{\BMO_\mathscr{B}}^k\, \|f\|_{L^p (w^p)}.
\end{equation}
\end{theorem}

The proof is exactly as before with the only exception that in Lemma \ref{lemma:BMO->Ap} when $\lambda=1$ one has that $w\in A_{q,\mathscr{B}}$ for every $q>1$ (recall that we have not defined $A_{1,\mathscr{B}}$) and also that, in the second part of that result, the case $p=1$ is excluded. Taking all this into account, one can follow the proof of Theorem \ref{thm:main:I} line by line and the argument goes through with only notational changes.

\begin{remark}
If one forgets about the dependence on the constants, Theorem \ref{thm:main:B} can be re-stated in a qualitative way. Such a result will say that if $T$ is bounded from $L^p(w^p)$ to $L^q(w^q)$ for every $w$ such that $w^\theta \in A_{s,\mathscr{B}}$ then the commutators $T_b^k$ with $b\in \BMO_{\mathscr{B}}$ are bounded from $L^p(w^p)$ to $L^q(w^q)$ for every $w$ such that $w^{\theta\,\eta} \in A_{s,\mathscr{B}}$ for some $\eta>1$. 
\end{remark}

Examples of bases $\mathscr{B}$ in $\mathbb{R}^n$ with the Lebesgue measure are the sets of all cubes, or balls, or dyadic cubes or rectangles with sides parallel to the axes. In spaces of homogeneous type, one can also take balls, or Christ's dyadic cubes. All these classes have some additional features that allow one to refine the statement of Theorem \ref{thm:main:B}. We will discuss this in the following section.

\subsection{Commutators in the classical setting II}\label{section:Comm-class-II}
Our next goal is to sharpen Theorem \ref{thm:main:I} and exploit the fact that the Muckenhoupt weights have self-improving properties. We first notice that in that result the weighted estimates for $T$ are assumed to hold for every weight $w$ in the class $w^{\theta}\in A_s$. However, in the conclusion, the weighted estimates for the commutators hold provided $w^{\theta\,\eta}\in A_s$, where $\eta>1$, which is apparently stronger than $w^{\theta}\in A_s$. Nevertheless, the self-improvement of the weights gives that $w^{\theta}\in A_s$ which implies $w^{\theta\,\eta}\in A_s$ for some $\eta>1$ which is typically very close to $1$. This value of $\eta$ depends on $s$ and $\big[w^\theta \big]_{A_s}$ and has been precisely quantified in \cite{Perez-course}. Let us notice again that this property depends very much on the setting. In the abstract context of the previous section without any further assumptions there is no reason to think that $w^{\theta}\in A_s$ implies $w^{\theta\,\eta}\in A_s$.

\begin{theorem} \label{thm:main:II}
Let $T$ be either a linear or a linearizable operator. Fix $1 \leq p, q < \infty$, $\theta > 0$, and $1 < s< \infty$. Suppose that  there exist  an increasing function $\phi: [1, \infty) \longrightarrow  [0, \infty)$ such that for each $f \in L^p (w^p)$ we have
\begin{equation}\label{eqn:main:hyp:II}
\|Tf\|_{L^q (w^q)}
\leq
\phi \left( \big[w^\theta\big]_{A_s} \right) \|f\|_{L^p (w^p)},
\qquad
\forall\, w^\theta \in A_s.
\end{equation}
Then, for each $b \in  \BMO$ and for every $w^\theta \in A_s$, we have
\begin{align}\label{eqn:main:k=1:II}
\big\|[T, b]f   \big\|_{L^q (w^q)}
&\leq
\frac{2^{2\,\max\{s,s'\}+n+2}\,\theta}{\min \left\{ 1, s-1 \right\}}\,
\,\big[w^\theta\big]_{A_s}^{\max\{1,\frac1{s-1}\}}\, \\ \nonumber
&\qquad
\phi
\left( 4^{ \min\{1,s-1\}}\,2^s\, \big[w^{\theta}\big]_{A_s}\right)\,
\|b\|_{ \BMO}\, \|f\|_{L^p (w^p)}
;
\end{align}
and, for every $k\ge 2$,
\begin{align}\label{eqn:main:k:II}
\big\|T_b^k f \big\|_{L^q (w^q)}
&\leq
k! \left(\frac{2^{2\,\max\{s,s'\}+n+2}\,\theta}{\min \left\{ 1, s-1 \right\}}\right)^k\,
\big[w^\theta\big]_{A_s}^{k\,\max\{1,\frac1{s-1}\}}\, \\ \nonumber
 &\qquad \phi\left( 4^{ \min\{1,s-1\}}\,2^s\, \big[w^{\theta}\big]_{A_s}\right)\,
\|b\|_{ \BMO}^k\, \|f\|_{L^p (w^p)}.
\end{align}
\end{theorem}

Before proving this result, we need the following auxiliary lemma due to P\'erez \cite{Perez-course}.

\begin{lemma}\label{RH-sharp}
Let $w\in A_p$, $1<p<\infty$, and let $\rho_w=1+\frac1{2^{2\,p+n+1}\,[w]_{A_p}}$. Then,  for any cube $Q\subset \mathbb R^n$, we have
\begin{equation}
\left( \fint_Q w^{\rho_w} \,dx \right)^{\frac1{\rho_w}}
\le
2\,\fint_Q w \,dx.
\end{equation}
\end{lemma}

\begin{proof}[Proof of Theorem \ref{thm:main:II}]
Fix $u=w^\theta\in A_s$. Write $v=u^{1-s'}=w^{\theta\,(1-s')}\in A_{s'}$. Apply Lemma \ref{RH-sharp} to $u$ and $v$ to obtain that
\begin{equation}\label{eqn:u-v}
\left( \fint_Q u^{\rho_u} \,dx \right)^{\frac1{\rho_u}}
\le
2\,\fint_Q u\,dx,
\qquad\quad
\left( \fint_Q v^{\rho_v} \,dx \right)^{\frac1{\rho_v}}
\le
2\,\fint_Q v \,dx,
\end{equation}
where
$$
\rho_u
=
1+\frac1{2^{2\,s+n+1}\,[u]_{A_s}}
,
\qquad\quad
\rho_v
=
1+\frac1{2^{2\,s'+n+1}\,[v]_{A_{s'}}}
=
1+\frac1{2^{2\,s'+n+1}\,[u]_{A_{s}}^{s'-1}}.
$$
Set $\eta=\min\{\rho_u, \rho_v\}$. We see that $u^\eta \in A_s$ by \eqref{eqn:u-v} and our choice of $\eta$:
\begin{multline*}
\big[u^\eta\big]_{A_s}^{\frac1\eta}
=
\sup_Q
\left( \fint_Q u^\eta\,dx \right)^{\frac1\eta} \left( \fint_Q (u^\eta)^{1- s'} \,dx \right)^{\frac{s-1}{\eta}}
\\
\le
\sup_Q
\left( \fint_Q u^{\rho_u}\,dx \right)^\frac{1}{\rho_u} \left( \fint_Q v^{\rho_v} \,dx \right)^{\frac{s-1}{\rho_v}}
\\
\le\sup_Q
2^{s}\left( \fint_Q u\,dx \right)\left( \fint_Q v\,dx \right)^{(s-1)}
=
2^{s}\,[u]_{A_s}.
\end{multline*}
We now estimate $\eta'$. We first observe that
$$
\eta'
=
\max\{\rho_u',\rho_v'\}
=
1+\max\big\{2^{2\,s+n+1}\,[u]_{A_s}, 2^{2\,s'+n+1}\,[u]_{A_{s}}^{s'-1}
\big\}.
$$
If $1<s\le 2$ we have that $2^{2\,s+n+1}\,[u]_{A_s}\le 2^{2\,s'+n+1}\,[u]_{A_{s}}^{s'-1}$ since $1\le s'-1$ and $[u]_{A_s}\ge 1$. On the other hand, if
$2\le s<\infty$ we analogously observe that $2^{2\,s'+n+1}\,[u]_{A_{s}}^{s'-1}\le 2^{2\,s+n+1}\,[u]_{A_s}$. Thus,
$$
\eta'
=
1+2^{2\,\max\{s,s'\}+n+1}\,[u]_{A_s}^{\max\{1,s'-1\}}
=
1+2^{2\,\max\{s,s'\}+n+1}\,[u]_{A_s}^{\max\{1,\frac1{s-1}\}}.
$$
By  \eqref{eqn:main:k=1} in Theorem \ref{thm:main:I}, the previous estimates, and the fact that $\eta'\ge 1$ we obtain
\begin{align*}
\big\|[T,b]f \big\|_{L^q (w^q)}
 & \leq
\frac{\eta'\,\theta}{\min \left\{ 1, s-1 \right\}}\,
\phi
\left( 4^{ \frac{\min\{1,s-1\}}{\eta'}}\, \big[w^{\theta\,\eta}\big]_{A_s}^{1/\eta} \right)\,
\|b\|_{ \BMO}\, \|f\|_{L^p (w^p)}
\\
 & \le
\frac{\theta}{\min \left\{ 1, s-1 \right\}}\,\big(1+2^{2\,\max\{s,s'\}+n+1}\,[u]_{A_s}^{\max\{1,\frac1{s-1}\}}\big)\,\\
& \hspace{10mm}
\phi
\left( 4^{ \min\{1,s-1\}}\,2^s\, \big[w^{\theta}\big]_{A_s}\right)\,
\|b\|_{ \BMO}\, \|f\|_{L^p (w^p)}
\\
 & \le
\frac{2^{2\,\max\{s,s'\}+n+2}\,\theta}{\min \left\{ 1, s-1 \right\}}\,
\,[u]_{A_s}^{\max\{1,\frac1{s-1}\}}\,\\
& \hspace{10mm} \phi
\left( 4^{ \min\{1,s-1\}}\,2^s\, \big[w^{\theta}\big]_{A_s}\right)\,
\|b\|_{ \BMO}\, \|f\|_{L^p (w^p)}.
\end{align*}
The proof of \eqref{eqn:main:k:II} is analogous and we leave it to the interested reader.
\end{proof}

\subsection{Commutators vs weighted norm inequalities}
In Theorem \ref{thm:main:I} we proved that if a given linear operators $T$ satisfies weighted norm inequalities, then so do the corresponding commutators with $\BMO$ functions. In many situations, the theories of weighted norm inequalities and the estimates for the commutators run parallel without apparent  connection. Theorem \ref{thm:main:I} gives a very powerful one-way bridge between the two theories. It is therefore natural to wonder whether there is in fact a two-way bridge between the two theories, that is, also an abstract principle that allows one to say that estimates for the commutators imply certain norm inequalities.

Suppose that a given operator $T$ has the property that every commutator $T_b^k$ is bounded on $L^2(\mathbb{R}^n)$ with an appropriate control on the constants. We wonder whether this implies that $T$ is bounded on $L^2(w)$ for some collection of weights. The best that one could expect the class of weights to be is $A_2$ (or some $A_p$), since we know that for the Riesz or Hilbert transforms the condition $A_2$ is necessary and sufficient. However, this class is too big. In fact, there are examples of operators whose commutators behave as the ones of the Riesz or Hilbert transforms for which the class of weights cannot be $A_2$. For example, consider the Riesz transform $\nabla L^{-1/2}$ where $Lf=-\text{div}(A\nabla f)$ is a second order divergence form elliptic operator with bounded and complex matrix $A$. All commutators of $L$ are bounded on $L^2(\mathbb{R}^n)$ with the appropriate constants; see \cite{AM3} and Section \ref{applications} below. However, if $\nabla L^{-1/2}$ were bounded on $L^2(w)$ for every $w\in A_2$, then the Rubio de Francia extrapolation theorem would imply that  $\nabla L^{-1/2}$ is bounded  on $L^p(\mathbb{R}^n)$ for every $1<p<\infty$ and this is not true for all such operators $L$; see \cite{Aus}. This shows that the class of weights considered for such a converse should be small enough  so that extrapolation does not lead to invalid estimates.

\begin{theorem}\label{thm:Comm-vs-Weights}
Fix $1 < p < \infty$ and suppose that $T$ is a linear operator.

\begin{list}{$(\theenumi)$}{\usecounter{enumi}\leftmargin=1cm \labelwidth=1cm \itemsep=0.2cm \topsep=0cm \renewcommand{\theenumi}{\alph{enumi}}}
\item Suppose there is a $\lambda_0>0$ and $\phi: [0, \lambda_0] \longrightarrow [0, \infty)$ an increasing function such that for each $b \in  \BMO$ we have
\begin{equation}\label{thm:Comm-vs-Weights:a-Hyp}
\|Tf\|_{L^p (e^{\lambda\, b})} \leq \phi\big(|\lambda|\,\|b\|_{\BMO}\big)\, \|f\|_{L^p (e^{\lambda b})},
\qquad \forall\,\lambda,\ |\lambda|\le \frac{\lambda_0}{\|b\|_{\BMO}}.
\end{equation}
Then, for each $k \geq 0$ and $b\in  \BMO$, one has
$$\|T_b^k f\|_{L^p (\mathbb{R}^n)}
\leq
\phi(\lambda_0)\, k!\, \left( \frac{p}{\lambda_0} \right)^k \,\|b\|_{\BMO}^k\, \|f\|_{L^p (\mathbb{R}^n)}.
$$

\item Conversely, suppose that there exists $\lambda_0, C_0>0$ such that, for every $b\in \BMO$ and for every $k\ge 0$, one has
\begin{equation}\label{thm:Comm-vs-Weights:b-Hyp}
\|T_b^k f\|_{L^p(\mathbb{R}^n)} \leq C_0\,k! \left( \frac{p}{\lambda_0} \right)^k \,\|b\|_{ \BMO}^k\, \|f\|_{L^p(\mathbb{R}^n)}.
\end{equation}
Then, there exists an increasing function  $\phi: [0, \lambda_0) \longrightarrow [0, \infty)$ such that for every $b\in \BMO$ we have
\begin{equation}\label{thm:Comm-vs-Weights:b-Con}
\|Tf\|_{L^p (e^{\lambda b})}
\leq
\phi\big(|\lambda|\,\|b\|_{\BMO}\big)\,\|f\|_{L^p (e^{\lambda b})},
\qquad \forall\,\lambda,\ |\lambda|< \frac{\lambda_0}{\|b\|_{\BMO}}.
\end{equation}
\end{list}
\end{theorem}

\begin{proof}
We shall first prove $(a)$. The case $k=0$ follows from \eqref{thm:Comm-vs-Weights:a-Hyp} with $\lambda=0$. Let $k\ge 1$. By homogeneity we may assume that $\|b\|_{\BMO} = 1$. We set $\delta=\lambda_0/p$. Write, as before $\Psi(z)f(x): = e^{-z\,b(x)} T (e^{z\,b}\,f)(x)$. The Cauchy Integral Formula gives
$$
T_b^k f(x)
=
\Psi^{(k)}(0)f(x)
= \frac{k!}{2\pi i} \int_{|z|=\delta} \frac{\Psi(z)f(x)}{z^{k+1}}\, dz.
$$
For every $|z|=\delta$ we note that $p\,|\Re(z)|\le p\,\delta\le \lambda_0$. Thus we can apply \eqref{thm:Comm-vs-Weights:a-Hyp} with $e^{-p\,\Re(z)\,b}$ and write
\begin{multline*}
\|T_b^k f\|_{L^p(\mathbb{R}^n)}
\leq \frac{k!}{2\,\pi\,\delta^{k+1}} \int_{|z| = \delta} \|\Psi (z) f\|_{L^p(\mathbb{R}^n)} |dz|
\le
\frac{k!}{\delta^{k}} \sup_{|z| = \delta} \|T(e^{z\,b}\,f)\|_{L^p(e^{-p\,\Re(z)\,b})}
\\
\le
\frac{k!}{\delta^k}\,\|f\|_{L^p (\mathbb{R}^n)}\,\sup_{|z| = \delta} \phi \big( p\,|\Re(z)|\big)
\le
\frac{k!}{\delta^k}\,\phi(\lambda_0),
\end{multline*}
which is the desired estimate.

Let us know turn to proving $(b)$. By homogeneity we may assume again that $\|b\|_{ \BMO} = 1$. Define $\Psi(z)$ as before and and fix $x$. Since $\Psi(z)f(x)$ is holomorphic, we can write
$$
|\Psi(z)f(x)|
=
\left|
\sum_{k=0}^\infty \frac{\Psi^{(k)}(0)f(x)}{k!}\,z^k
\right|
\le
\sum_{k=0}^\infty \frac{|T_b^k f(x)|}{k!}\,|z|^k.
$$
We now use \eqref{thm:Comm-vs-Weights:b-Hyp} to obtain that if $|z|<\lambda_0/p$ then
\begin{align*}
\|\Psi(z)f\|_{L^p(\mathbb{R}^n)}
 & \le
\sum_{k=0}^\infty \frac{1}{k!}\,|z|^k\,\|T_b^k f\|_{L^p(\mathbb{R}^n)}
\le
C_0\,\|f\|_{L^p(\mathbb{R}^n)}\,
\sum_{k=0}^\infty |z|^k\, \left( \frac{p}{\lambda_0} \right)^k  =
\phi(|z|\,p)\,
\|f\|_{L^p(\mathbb{R}^n)}\,
\end{align*}
where $\phi: [0, \lambda_0) \longrightarrow [0, \infty)$ is the increasing function defined by the formula $\phi(s)=C_0\,(1-s/\lambda_0)^{-1}$, $0\le s<\lambda_0$.
Let us now take $z=-\lambda/p$ with $|\lambda|<\lambda_0$ (which implies that $|z|<\lambda_0/p$) and $f(x)=g(x)\,e^{\lambda\,b(x)/p}$. The previous estimate yields
$$
\|Tg\|_{L^p(e^{\lambda\,b})}
=
\|\Psi(-\lambda/p)f\|_{L^p(\mathbb{R}^n)}
\le
\phi(|\lambda|)\,
\|f\|_{L^p(\mathbb{R}^n)}\,
=
\phi(|\lambda|)\,\|g\|_{L^p(e^{\lambda\,b})}.
$$
This completes the proof.
\end{proof}

\section{Multilinear operators}\label{multilinear}
We turn now our attention to multilinear operators with the main goal of obtaining the natural counterparts to the general results from the linear setting, Theorem \ref{thm:main:I} and Theorem \ref{thm:main:II}. The situation here changes slightly since the good classes of weights come in two flavors: product weights and vector weights.

Let us thus begin by defining the main objects that we will be dealing with in this setting. $T$ will always denote an $m$-linear operator from $X_1 \times \cdots \times X_m$ into $Y$, where $X_j, 1\le j\le m$, and $Y$ are some normed spaces. In our statements the $X_j$ and $Y$ will be again appropriate weighted Lebesgue spaces. For $\textbf{f}=(f_1, f_2,\dots, f_m)\in
X_1\times X_2\times\cdots \times X_m$ and for a measurable vector $\textbf{b}=(b_1, b_2, \dots, b_m)$, and $1\leq j\leq m$, we define, whenever it makes sense, the (first order) commutators
$$[T, \textbf{b}]_{e_j}({\textbf f})=b_jT(f_1,\dots, f_j,\dots f_m)-T(f_1,\dots, b_jf_j,\dots f_m);$$
we denoted by $e_j$ the basis element taking the value $1$ at component $j$ and $0$ in every other component, therefore expressing the fact that the commutator acts as a linear one in the $j$th variable and leaving the rest of the entries of $\textbf{f}$ untouched. Then, if $k\in\mathbb N$, we define
$$[T, \textbf{b}]_{k e_j} = [ \cdots [ [ T, \textbf{b}]_{e_j}, \textbf{b}]_{e_j} \cdots, \textbf{b}]_{e_j},$$ where the commutator is performed $k$ times. Finally, if
$\alpha = (\alpha_1, \alpha_2,\dots, \alpha_m) \in \mathbb (\mathbb N\cup \{0\})^m$ is a multi-index, we define
\begin{align}
[T, \textbf{b}]_\alpha = [ \cdots [ [T, \textbf{b}]_{\alpha_1 e_1}, \textbf{b}]_{\alpha_2 e_2} \cdots, \textbf{b}]_{\alpha_m e_m}.
\end{align}
Informally, if the multilinear operator $T$ has a kernel representation of the form
$$T(\textbf{f})(x)=\int_{\mathbb R^{nm}}K(x, y_1,\dots, y_m)f_1(y_1)\cdots f_m(y_m)dy_1\dots dy_m,$$
then $[T, \textbf{b}]_\alpha(\textbf{f})(x)$ can be expressed in a similar way, with kernel
$$\prod_{j=1}^m (b_j(x)-b_j(y_j))^{\alpha_j}K(x, y_1,\dots, y_m).$$
Next, we define the appropriate vector-weights classes of Lerner, Ombrosi, P\'erez, Torres and Truji\-llo-Gonz\'alez \cite{Lerner2009}.
\begin{definition}\label{vector-class}
Let $\textbf{P} = (p_1, \dots, p_m)$ with $1 < p_1, \dots, p_m < \infty$ satisfying
$$\frac{1}{p_1} + \cdots + \frac{1}{p_m} = \frac{1}{p}.$$
Let $\textbf{w} = (w_1, \dots, w_m)$ and set
$$\nu_{\textbf{w}} = \prod_{j=1}^m w_j^{p/p_j}.$$
We say that the vector weight $\textbf{w} \in A_{\textbf{P}}$ if
$$[\textbf{w}]_{A_{\textbf{P}}} := \sup_Q \left( \fint_Q \nu_w \,dx \right)^{1/p} \prod_{j=1}^m \left( \fint_Q w_j^{1-p_j'} \,dx \right)^{1/p_j'} < \infty.$$
\end{definition}
We observe that, by H\"older's inequality, we have
$$\left( \fint_Q \nu_{\textbf{w}} \,dx \right)^{1/p} \prod_{j=1}^m \left( \fint_Q w_j^{1-p_j'} \,dx \right)^{1/p_j'}  \!\! \leq \prod_{j=1}^m \left( \fint_Q w_j \,dx \right)^{1/p_j}  \!\! \left( \fint_Q w_j^{1-p_j'} \,dx \right)^{1/p_j'}\!.$$
and thus
$$\prod_{j=1}^m A_{p_j} \subset A_{\textbf{P}}.$$
It was shown in \cite{Lerner2009} that the above inclusion is strict.

\subsection{Multilinear commutators I} We are now able to give a multilinear version of Theorem \ref{thm:main:I} where each weight belongs, individually, to a Muckenhoupt class.

\begin{theorem}\label{thm:multilinear:main:I}
Let $T$ be an $m$-linear operator. Fix $1 < p_j< \infty$, $1 <s_j< \infty$, and $\theta_j > 0$, $1\le j\le m$. Let $1<p<\infty$ be so that
$$\frac{1}{p}=\sum_{j=1}^m \frac{1}{p_j},$$
and suppose that there exist increasing functions $\phi_j: [1, \infty)\to [0, \infty)$ such that for all $\textbf{w} = (w_1, \dots, w_m)$ satisfying $w_j^{\theta_j}\in A_{s_j}, 1\leq j\leq m$, we have
\begin{equation}
\label{product-weight}
\|T\textbf{f}\|_{L^p (\prod_{j=1}^m w_j^p)} \leq \prod_{j=1}^m \phi_j\left([w_j^{\theta_j}]_{A_{s_j}}\right)\|f_j\|_{L^{p_j}\left(w_j^{p_j}\right)}.
\end{equation}
Then, for all $\textbf{b} = (b_1, \dots, b_m) \in  \BMO^m$, for each multi-index $\alpha$ and vector weight $\textbf{w} = (w_1, \dots, w_m)$ such that for all $1\le j\le m$, $w_j^{\eta_j \theta_j} \in A_{s_j}$ with some $1<\eta_j<\infty$, we have
\begin{equation}
\label{multi-commutator-I}
\|[T, \textbf{b}]_\alpha \textbf{f}\|_{L^p (\prod_{j=1}^m w_j^p)}
\leq
\alpha!\prod_{j=1}^m \frac1{\delta_j^{\alpha_j}}\phi_j
\left(4^{\theta_j\delta_j}[w_j^{\theta_j\eta}]^{1/\eta}_{A_{s_j}}\right)\|b_j\|^{\alpha_j}_{\BMO} \|f_j\|_{L^{p_j}\left(w_j^{p_j}\right)},
\end{equation}	
where $\delta_j=\frac{\min\{1, s_j-1\}}{\eta_j'\theta_j}$, $1\le j\le m$.
\end{theorem}

\begin{proof}
We adapt the proof of Theorem \ref{thm:main:I} to the multilinear setting. Let then $\textbf{b} \in  \BMO^m$, and without loss of generality assume that they are real valued and $\|b_j\|_{\BMO} = 1, 1\le j\le m.$

Define, as before, for $z=(z_1,\ldots,z_n)$, the following holomorphic map
$$\Psi(z) (\textbf{f}) = e^{-\sum_{j=1}^m z_jb_j}T \left(e^{z_1 b_1} f_1, e^{z_2 b_2} f_2, \dots, e^{z_m b_m} f_m \right).$$
The Cauchy integral formula adapted to several complex variables now yields
\begin{equation}\label{multicauchy}
[T, \textbf{b}]_{\alpha}\textbf f
=
\left. D^\alpha\Psi_z(\textbf f)\right|_{z=0}
=
\frac{\alpha !}{(2\pi i)^m} \int_{b_0 P(0, \vec{\delta})} \frac{\Psi_z (\textbf{f})}{z_1^{\alpha_1+1}z_2^{\alpha_2 + 1}\cdots z_m^{\alpha_m+1}}\, dz_1dz_2\cdots dz_m,
\end{equation}
where we denoted by $b_0 P(0, \vec{\delta})=\{(z_1, z_2,\dots, z_m)\in \mathbb C^m: |z_j|=\delta_j, 1\le j\le m\}$ the so-called distinguished boundary of the polydisc
$P(0, \vec{\delta})=\{(z_1, z_2,\dots, z_m)\in \mathbb C^m: |z_j|<\delta_j, 1\le j\le m\}.$

Fix now the weights $w_j$ so that $w_j^{\theta_j\eta_j}\in A_{s_j}$ for $1\le j\le m$ and let us define $v_j=w_je^{-\text{Re}(z_j)b_j}$. Note that
$$v_j^{\theta_j}=w_j^{\theta_j} e^{-\text{Re}(z_j)\theta_jb_j}.$$
Since $w_j^{\theta_j\eta_j}\in A_{s_j}$ and
$$
|\text{Re}(z_j)\theta_j|\leq \delta_j\theta_j\le \frac{\min\{1, s_j-1\}}{\eta_j'},
$$
by Lemma \ref{lemma:BMO->Ap:prod} we obtain that $v_j^{\theta_j}\in A_{s_j}$ and
$$
[v_j^{\theta_j}]_{A_{s_j}}\le [w_j^{\theta_j\eta_j}]_{A_{s_j}}^{1/\eta_j}4^{\delta_j \theta_j}.
$$
By our hypothesis condition \eqref{product-weight}, we can write
\begin{align*}
\|\Psi_z \textbf{f}\|_{L^p (\prod_{j=1}^m w_j^p)}&=\|T(e^{z_1b_1}f_1,\dots , e^{z_m b_m}f_m)\|_{L^p (\prod_{j=1}^m v_j^p)}\\
&\leq \prod_{j=1}^m \phi_j\left([v_j^{\theta_j}]_{A_{s_j}}\right)\|e^{z_jb_j}f_j\|_{L^{p_j}\left(v_j^{p_j}\right)},
\end{align*}
which, by the increasing property of the functions $\phi_j$ and the fact that
$$
\|e^{z_jb_j}f_j\|_{L^{p_j}\left(v_j^{p_j}\right)}=\|f_j\|_{L^{p_j}\left(w_j^{p_j}\right)}
$$
further yields the uniform estimate over $b_0P(0, \vec\delta)$:
$$
\|\Psi_z \textbf{f}\|_{L^p (\prod_{j=1}^m w_j^p)}
\le \prod_{j=1}^m \phi_j
\left(4^{\theta_j\delta_j}[w_j^{\theta_j\eta_j}]^{1/\eta_j}_{A_{s_j}}\right) \|f_j\|_{L^{p_j}\left(w_j^{p_j}\right)}.
$$
Now, by Minkowski's inequality, we obtain that
\begin{align*}
\|[T, \textbf{b}]_{\alpha}\textbf f\|_{L^p (\prod_{j=1}^m w_j^p)}
&\leq
\frac{\alpha!}{(2\pi)^m}\int_{b_0 P(0, \vec{\delta})} \frac{\|\Psi_z \textbf{f}\|_{L^p (\prod_{j=1}^m w_j^p)}}{\prod_{j=1}^m |z_j|^{\alpha_j+1}} |dz_1|\cdots |dz_m|
\\
&\leq \frac{\alpha!}{(2\pi)^m}\prod_{j=1}^m \frac{2\pi\delta_j}{\delta_j^{\alpha_j+1}} \phi_j
\left(4^{\theta_j\delta_j}[w_j^{\theta_j\eta_j}]^{1/\eta}_{A_{s_j}}\right) \|f_j\|_{L^{p_j}\left(w_j^{p_j}\right)},
\end{align*}
which is \eqref{multi-commutator-I}.
\end{proof}

The statement of Theorem \ref{thm:multilinear:main:I} and its argument can now be adapted in a straightforward way to obtain a multilinear version of Theorem \ref{thm:main:II}. We leave this task to the interested reader. However, we would like to make several remarks on Theorem \ref{thm:multilinear:main:I}.

By re-normalizing the weights and writing $w_j:=w_j^{p_j}$, the statement of our theorem can be reformulated in terms of the weight $\nu_{\textbf{w}}$ introduced in Definition \ref{vector-class}. Ignoring the precise constants in terms of the increasing functions $\phi_j$ in Theorem \ref{thm:multilinear:main:I}, we obtain the following.

\begin{corollary}
\label{cor:multilinear:I}
Let $T$ be an $m$-linear operator. Fix $1 < p_j< \infty$, $1 <s_j< \infty$, and $\theta_j > 0$, $1\le j\le m$. Let $1<p<\infty$ be so that
$$\frac{1}{p}=\sum_{j=1}^m \frac{1}{p_j},$$
and suppose that for all $\textbf{w} = (w_1, \dots, w_m)$ satisfying $w_j^{\theta_j/p_j}\in A_{s_j}, 1\leq j\leq m$, we have
\begin{equation}
\label{product-weight:v1}
\|T\textbf{f}\|_{L^p (\nu_{\textbf{w}})} \lesssim \prod_{j=1}^m \|f_j\|_{L^{p_j}\left(w_j\right)}.
\end{equation}
Then, for all $\textbf{b} = (b_1, \dots, b_m) \in  \BMO^m$, for each multi-index $\alpha$ and vector weight $\textbf{w} = (w_1, \dots, w_m)$ such that for all $1\le j\le m$, $w_j^{\eta_j \theta_j/p_j} \in A_{s_j}$ with some $1<\eta_j<\infty$, we have
\begin{equation}
\label{multi-commutator:v1}
\|[T, \textbf{b}]_\alpha \textbf{f}\|_{L^p (\nu_{\textbf{w}})} \lesssim  \prod_{j=1}^m \|b_j\|^{\alpha_j}_{\BMO} \|f_j\|_{L^{p_j}\left(w_j\right)}.
\end{equation}
\end{corollary}
The following is a variation on the corollary just stated, in which we express the conditions on the weights in a simpler manner.

\begin{corollary}
\label{cor:multilinear:II}
Let $T$ be an $m$-linear operator. Fix $1 < p_j< \infty$, $1\le j\le m$, and $\theta > 0$. Let $1<p<\infty$ be so that
$$\frac{1}{p}=\sum_{j=1}^m \frac{1}{p_j},$$
and suppose that for all $\textbf{w} = (w_1, \dots, w_m)$ satisfying $w_j^{\theta}\in A_{p_j}, 1\leq j\leq m$, we have
\begin{equation}
\label{product-weight:v2}
\|T\textbf{f}\|_{L^p (\nu_{\textbf{w}})} \lesssim \prod_{j=1}^m \|f_j\|_{L^{p_j}\left(w_j\right)}.
\end{equation}
Then, for all $\textbf{b} = (b_1, \dots, b_m) \in  \BMO^m$, for each multi-index $\alpha$ and vector weight $\textbf{w} = (w_1, \dots, w_m)$ as above, we have
\begin{equation}
\label{multi-commutator:v11}
\|[T, \textbf{b}]_\alpha \textbf{f}\|_{L^p (\nu_{\textbf{w}})} \lesssim  \prod_{j=1}^m \|b_j\|^{\alpha_j}_{\BMO} \|f_j\|_{L^{p_j}\left(w_j\right)}.
\end{equation}
\end{corollary}


\subsection{Multilinear commutators II} The unsatisfying aspect of the statement of Theorem \ref{thm:multilinear:main:I} or of its subsequent corollaries is that the dependence of the constants in the estimates is on the Muckenhoupt constant of each one of the weights $w_j$ that comprise the vector weight $\textbf{w}$. We present here a version with constants depending only on the $A_{\textbf P}$ constant of $\textbf{w}$.

\begin{theorem}\label{thm:multilinear:main:II}
Let $T$ be an $m$-linear operator. Fix $1 < p_j< \infty$ and let $1<p<\infty$ be so that
$$\frac{1}{p}=\sum_{j=1}^m \frac{1}{p_j}.$$
Suppose that there exists an increasing functions $\phi: [1, \infty)\to [0, \infty)$ such that for all $\textbf{w} = (w_1, \dots, w_m)\in A_{\textbf{P}}$, we have
\begin{equation}
\label{vector-weight}
\|T\textbf{f}\|_{L^p (\nu_{\textbf{w}})} \lesssim \phi \left([\textbf{w}]_{A_{\textbf P}}\right)\prod_{j=1}^m \|f_j\|_{L^{p_j}\left(w_j\right)}.
\end{equation}
Then, for all $\textbf{b} = (b_1, \dots, b_m) \in  \BMO^m$ and for each multi-index $\alpha$, we have
\begin{equation}
\label{multi-commutator-II}
\|[T, \textbf{b}]_\alpha \textbf{f}\|_{L^p (\nu_{\textbf{w}})}
\lesssim
\alpha !\,\phi\left(c_{\textbf{P}} [\textbf{w}]_{A_{\textbf P}}\right)
[\bw]_{A_\bp}^{|\alpha|\max\{p,p_1',\ldots,p_m'\}}
\prod_{j=1}^m \|b_j\|^{\alpha_j}_{\BMO} \|f_j\|_{L^{p_j}\left(w_j\right)},
\end{equation}
where $c_{\textbf{P}}=4^{1+\sum_{j=1}^m \min\{1/p_j, 1/p_j'\}}$.
\end{theorem}

For the convenience of notation, we will only prove this result for bilinear operators ($m=2$). Before we begin our argument, let us recall that if $\textbf{w}=(w_1, w_2)$ and $\textbf{P}=(p_1, p_2)$, then we have \cite{Lerner2009} the following equivalence:
$$\textbf{w}\in A_{\textbf{P}}\Leftrightarrow \nu_{\textbf{w}}=w_1^{p/p_1}w_2^{p/p_2}\in A_{2p}, \sigma_1:=w_1^{1-p_1'}\in A_{2p_1'},\,\,\text{and}\,\, \sigma_2:=w_2^{1-p_2'}\in A_{2p_2'}.$$
Moreover, by \cite[Lemma 3.1]{DLP2015}, we have $[\sigma_j]_{A_{2p_j'}}\leq [\bw]_{A_{\textbf P}}^{p_j'}, j=1, 2$, and since
$p/p_1'+p/p_2'=2p-1$, H\"older's inequality gives
\begin{multline*}\left(\fint_Q \nu_\bw\, dx\right)\left(\fint_Q\nu_\bw^{1-(2p)'}\,dx\right)^{2p-1}
\\
\leq \left(\fint_Q \nu_\bw\,dx\right)\left(\fint_Q \sigma_1\,dx\right)^{p/p_1'}\left(\fint_Q \sigma_2\,dx\right)^{p/p_2'}\leq [\bw]_{A_\bp}^p,
\end{multline*}
that is, $[\nu_\bw]_{A_{2p}}\leq [\bw]_{A_\bp}^p$.

\begin{proof}
As before, and without loss of generality, we assume that $b_j, j=1, 2$ are real valued and normalized so that their $\BMO$ norms are equal to 1. By repeating the argument with the Cauchy integral trick in Theorem \ref{thm:multilinear:main:I}, given $\textbf{w}\in A_{\textbf{P}}$ we see that everything reduces to showing that for some appropriate $\delta_1,\delta_2 >0$ (to be chosen later) and for $|z_1|=\delta_1$, $|z_2|=\delta_2$, we have
$$
\widetilde{\textbf{w}}:=
(\widetilde{w}_1,\widetilde{w}_2)
:=
(w_1 e_{b_1}, w_2 e_{b_2})
:=
(w_1 e^{-\text{Re}(z_1)p_1 b_1} , w_2 e^{-\text{Re}(z_2)p_2 b_2}) \in A_\bp.
$$
Note that as observed above $\nu_\bw\in A_{2p}$ and $\sigma_j:=w_j^{1-p_j'}\in A_{2p_j'}$, for $j=1,2$. Using now Lemma \ref{RH-sharp} and writing, for a given weight $w$, $\rho(w)$ instead of $\rho_w$, we can find $r=r(\bw)=\min \{\rho (\nu_\bw), \rho (\sigma_1),\rho (\sigma_2) \}>1$ so that
\begin{equation}\label{eq:afrfr}
r'\sim \max\{[\nu_\bw]_{A_{2p}}, [\sigma_1]_{A_{2p_1}}, [\sigma_2]_{A_{2p_2}}\}\le
[\bw]_{A_\bp}^{\max\{p,p_1',p_2'\}}
\end{equation}
and the following reverse H\"older inequalities hold:
\begin{equation}\label{rh1}
\left( \fint_Q \nu_\bw^{r}\,dx \right)^{1/r} \leq 2 \fint_Q\nu_\bw \,dx
\end{equation}
and for $j=1,2$
\begin{equation}\label{rh2}
\left( \fint_Q \sigma_j^{r}\,dx \right)^{1/r} \leq 2 \fint_Q \sigma_j\,dx.
\end{equation}
Using these, H\"older's inequality and regrouping terms, we get
\begin{align*}
&\left( \fint_Q \nu_{\widetilde{\bw}}\,dx \right)^{1/p}
\
\prod_{j=1}^2 \left( \fint_Q \widetilde{w}_j^{1-p_j'} \,dx \right)^{1/p_j'}
\\
&\qquad\qquad=
\left( \fint_Q \nu_{\bw}\,e_{b_1}^{p/p_1}\,e_{b_2}^{p/p_1}\,dx \right)^{1/p}
\prod_{j=1}^2 \left( \fint_Q \sigma_j \,e_{b_j}^{1-p_j'}\,dx \right)^{1/p_j'}
\\
&\qquad\qquad\le
\left( \fint_Q \nu_{\bw}^{r}\,dx \right)^{1/(p r)}
\left( \fint_Q e_{b_1}^{r'p/p_1}\,e_{b_2}^{r'p/p_1}\,dx \right)^{1/(p r')}
\\
&\qquad\qquad\quad\qquad
\prod_{j=1}^2
\left( \fint_Q \sigma_j^{r}\,dx \right)^{1/(p_j'r)}
\,
\left( \fint_Q e_{b_j}^{r'(1-p_j')}\,dx \right)^{1/p_j'r'}
\\
&
\qquad\qquad\le
4[\bw]_{A_\bp}\prod_{j=1}^2 \left( \fint_Q e_{b_j}^{r'}\,dx \right)^{1/(p_j r')} \left( \fint_Q e_{b_j}^{r'(1-p_j')}\,dx \right)^{1/p_j'r'}
\\
&
\qquad\qquad\le
4[\bw]_{A_\bp}\prod_{j=1}^2 \left[e^{-\text{Re}(z_j)p_j r' b_j} \right]_{A_{p_j}}^{1/(p_j r')}
\\
&
\qquad\qquad\le
4^{1+\delta_1+\delta_2}[\bw]_{A_\bp},
\end{align*}
where the last estimates uses Lemma \ref{lemma:BMO->Ap} as long as we assume that
$\delta_j\leq \frac{\min\{1, p_j-1\}}{p_jr'}$, $j=1,2$. Notice that this choice implies that $\delta_j\le \min\{1/p_j, 1/p_j'\}$. On the other hand, recalling \eqref{eq:afrfr} and assuming further than
$\delta_j\sim [\bw]_{A_\bp}^{-\max\{p,p_1',p_2'\}}$ we eventually obtain
$$
\|[T, \textbf{b}]_\alpha \textbf{f}\|_{L^p (\nu_{\textbf{w}})} \lesssim \alpha !\delta_1^{-\alpha_1}\delta_2^{-\alpha_2}
\phi\left(c_{\textbf{P}}[\textbf{w}]_{A_{\textbf P}}\right)  \|f_1\|_{L^{p_1}\left(w_1\right)}\|f_2\|_{L^{p_2}\left(w_2\right)}.
$$
with $c_{\textbf{P}}=4^{1+\min\{1/p_1, 1/p_1'\}+\min\{1/p_2, 1/p_2'\}}$. This easily gives the desired estimate.
\end{proof}

\subsection{Multilinear commutators vs weighted estimates}
In this subsection we give a multilinear version of Theorem \ref{thm:Comm-vs-Weights},
thus showing that there is a two way bridge between weighted estimates for commutators of the operator and weighted estimates for the operator itself in the multilinear setting as well.  We do not include the proof as it follows the linear situation very closely using the Cauchy integral formula for several complex variables \eqref{multicauchy} and a multivariable Taylor series.

\begin{theorem}\label{thm:adedew}
Let ${\textbf P}=(p_1,\ldots,p_m)\in (1,\infty)^m$ and $p>1$ be such that $\frac1p=\frac1{p_1}+\cdots+\frac1{p_m}$, and let $T$ be a multilinear operator.
\begin{list}{$(\theenumi)$}{\usecounter{enumi}\leftmargin=1cm \labelwidth=1cm \itemsep=0.2cm \topsep=0cm \renewcommand{\theenumi}{\alph{enumi}}}
\item Suppose there are $\lambda_0>0$ and a function $\Phi:[0,\lambda_0]^m\rightarrow [0,\infty)$ that is increasing in each variable such that for each $\textbf{b}=(b_1,\ldots,b_m)\in \BMO^m$ we have
\begin{align*}
\|T\textbf{f}\|_{L^p\big(\exp(\sum_{j=1}^\infty \frac{p\lambda_j}{p_j}b_j)\big)}
\leq
\Phi(|\lambda_1|\,\|b_1\|_{\BMO},\ldots,|\lambda_m|\,\|b_m\|_{\BMO})\prod_{j=1}^m\|f_j\|_{L^{p_j}\big(\exp(\lambda_j b_j)\big)}
\end{align*}
for $|\lambda_j|\le\lambda_0/\|b_j\|_{\BMO}$.  Then, for each multi-index $\alpha$ and $\textbf{b}\in \BMO^m$
$$
\|[T,\textbf{b}]_\alpha\textbf{f}\|_{L^p(\R^n)}
\leq
\Phi(\lambda_0,\ldots,\lambda_0) \alpha!\prod_{j=1}^m\left(\frac{p_j\|b_j\|_{\BMO}}{\lambda_0}\right)^{\alpha_j} \|f_j\|_{L^{p_j}(\R^n)}.
$$
\item Conversely, suppose there exist $C_0,\lambda_0>0$ such that for $\textbf{b}=(b_1,\ldots,b_m)\in \BMO^m$ and for every multi-index $\alpha$ we have
$$
\|[T,\textbf{b}]_\alpha\textbf{f}\|_{L^p(\R^n)}
\leq
C_0 \alpha!\prod_{j=1}^m\left(\frac{p_j\|b_j\|_{\BMO}}{\lambda_0}\right)^{\alpha_j} \|f_j\|_{L^{p_j}(\R^n)}.
$$
Then, there exists a function $\Phi:[0,\lambda_0)^m\rightarrow [0,\infty)$, increasing in each variable, such that for every $\textbf{b}\in \BMO^m$ we have
\begin{align*}
\|T\textbf{f}\|_{L^p\big(\exp(\sum_{j=1}^\infty \frac{p\lambda_j}{p_j}b_j)\big)}\leq \Phi(|\lambda_1|\,\|b_1\|_{\BMO},\ldots,|\lambda_m|\,\|b_m\|_{\BMO})\prod_{j=1}^m\|f_j\|_{L^{p_j}\big(\exp(\lambda_j b_j)\big)},
\end{align*}
for all $\lambda_1,\ldots,\lambda_m$ satisfying $|\lambda_j|<\lambda_0/\|b_j\|_{\BMO}$.

\end{list}
\end{theorem}

After this work was completed we have learned that Kunwar and Ou \cite{KO} have extended some of our methods to the two-weight setting.

\section{Applications}\label{applications}
The goal of this section is to apply our general results connecting weighted estimates for an operator and those of its commutators in various situations, some explored and some unexplored before. We begin with the case of linear or linearizable operators.

\subsection{The Coifman-Rochberg-Weiss result}
As a first application, from Theorem \ref{thm:main:II} we recover the classical Coifman-Rochberg-Weiss (see \cite{CRW1976}) whose quantitative control of the constants appeared in \cite{CPP2012}.
\begin{corollary}
Let $T$ be either a linear or a linearizable operator. Fix $1<p<\infty$,  and suppose that  there exist  an increasing function $\phi: [1, \infty) \longrightarrow  [0, \infty)$ such that for each $f \in L^p (w)$ there holds
$$
\|Tf\|_{L^p (w)}
  \leq
\phi \left( [w]_{A_p} \right) \|f\|_{L^p (w)},
\qquad
\forall\, w\in A_p.
$$
Then, for each $b \in  \BMO$ and for every $w\in A_p$
and, for every $k\ge 1$, there holds
\begin{align*}
\big\|T_b^k f \big\|_{L^p (w)}
& \leq
k!\left(\frac{2^{2\,\max\{p,p'\}+n+2}\,p}{\min \left\{ 1, p-1 \right\}}\right)^k\,
[w]_{A_p}^{k\,\max\{1,\frac1{p-1}\}}\,
\phi
\left( 4^{ \min\{1,p-1\}}\,2^p\, [w]_{A_p}\right)\,\\
 & \hspace{15mm}
\|b\|_{ \BMO}^k\, \|f\|_{L^p (w)}.
\end{align*}
\end{corollary}

If we now take $T$ a Calder\'on-Zygmund operator, the $A_2$-conjecture proved by Hyt\"onen in \cite{Hytonen2012} gives the result in \cite{CPP2012}:
$$
\|Tf\|_{L^p (w)}
\leq
C\,[w]_{A_p}^{\max\{1,\frac1{p-1}\}} \|f\|_{L^p (w)},
\qquad
\forall\, w\in A_p,
$$
and therefore for the commutators one obtains
$$
\big\|T_b^k f \big\|_{L^p (w)}
\leq
C_k
[w]_{A_p}^{(k+1)\,\max\{1,\frac1{p-1}\}}\,
\|b\|_{ \BMO}^k\, \|f\|_{L^p (w)}.
$$

\subsection{Fractional integrals}
We begin by recalling the notion of the Riesz potentials or fractional integrals. For a more thorough introduction see \cite[Chapter 5]{Stein1970}.
Given $0 < \alpha < n$, we write
$$
I_\alpha f(x)=(-\Delta)^{-\alpha/2}f(x)=c_\alpha\,\int_{\mathbb{R}^n} \frac{f(y)}{|x-y|^{n - \alpha}}\, dy
$$
where $c_\alpha$ is a constant depending on $n$ and $\alpha$.
The relevant class of weights is $A_{p,q}$ defined as follows: $w \in A_{p, q}$  if
$$
[w]_{A_{p,q}}:=\sup_{Q\subset\mathbb{R}^n} \left( \fint_Q w^q \,dx \right)\left( \fint_Q w^{-p'} \,dx \right)^{q/p'} < \infty.
$$
The weighted norm inequalities for $I_\alpha$ were obtained by B. Muckenhoupt and R. Wheeden \cite{MW1974} and the sharp behavior in terms of the weight constants by \cite{Lacey2010}.  The precise estimate is as follows: for every $0 < \alpha < n$, $1 < p < n / \alpha$, $1/p - 1/q = \alpha / n$, and $w\in A_{p,q}$ one has
$$
\|I_\alpha f\|_{L^q (w^q)}
\leq
C_{p,q,\alpha}
[w]_{A_{p,q}}^{(1-\frac{\alpha}{n})\,\max\{1,\frac{p'}{q}\}}\, \|f\|_{L^p (w^p)}.
$$

It is easy to see that $w \in A_{p, q}$ iff $w^q\in A_{q\,\frac{n-\alpha}{n}}$ and moreover
$$
\big[w^q\big]_{A_{q\,\frac{n-\alpha}{n}}}
=
[w]_{A_{p,q}}.
$$
The commutators of fractional integrals and $BMO$ functions were first studied in \cite{SC1982}. We can use Theorem \ref{thm:main:II} and obtain the following quantitative weighted result which was shown in \cite{CM2012} for $k=1$, but seems to be new when $k>1$. \begin{corollary}\label{linear-fractional}
Fix $0 < \alpha < n$, $1 < p < n / \alpha$ and  $1/p - 1/q = \alpha / n$. For every $k\ge 1$ and $b\in  \BMO$ one has
$$
\big\|(I_\alpha)_b^k f \big\|_{L^q (w^q)}
\leq
C_{p,q,\alpha} k!
\left(\frac{2^{2\,\max\{s,s'\}+n+2}\,q}{\min \left\{ 1, s-1 \right\}}\right)^k\,
[w]_{A_{p,q}}^{(k+1-\frac{\alpha}{n})\,\max\{1,\frac{p'}{q}\}}\,
\|b\|_{ \BMO}^k\, \|f\|_{L^p (w^p)},
$$
for every $w\in A_{p,q}$ and where $s=q\,\frac{n-\alpha}{n}$.
\end{corollary}

\subsection{Operators associated with the Kato conjecture}
The following estimate models the behavior of operators associated with the Kato conjecture (see below).
\begin{corollary}\label{corol:restricted}
Let $1\le r_-<p<r_+\le\infty$ and assume that
$$
\|Tf\|_{L^p(w)}
\le
C_w\,\|f\|_{L^p(w)},
\qquad
\forall\,  w \in A_{\frac{p}{r_-}} \cap RH_{(\frac{r_+}{p})'}.
$$
Then, for every $b\in \BMO$ and for every $k\ge 1$
$$
\|T_b^kf\|_{L^p(w)}
\le
C_k\,\|b\|_{ \BMO}^k
\|f\|_{L^p(w)},
\qquad
\forall\,  w \in A_{\frac{p}{r_-}} \cap RH_{(\frac{r_+}{p})'}.
$$
\end{corollary}

Notice that if $r_+=\infty$ it is understood that the condition $RH_1$ is vacuous. This result follow easily from Theorem \ref{thm:main:II} (this time without paying attention to constants) after observing that Proposition \ref{prop:weights} (g) yields
$$
w \in A_{\frac{p_0}{p_-}} \cap RH_{\big(\frac{p_+}{p_0}\big)'}\qquad\Longleftrightarrow\qquad
w^{(\frac{p_+}{p_0})'}\in A_{\big(\frac{p_+}{p_0}\big)'\,\big(\frac{p_0}{p_-}-1\big)+1}.
$$

Let $A$ be an $n\times n$ matrix of complex and
$L^\infty$-valued coefficients defined on $\mathbb{R}^n$. We assume that
this matrix satisfies the following ellipticity (or \lq\lq
accretivity\rq\rq) condition: there exist
$0<\lambda\le\Lambda<\infty$ such that
$$
\lambda\,|\xi|^2
\le
\Re A(x)\,\xi\cdot\bar{\xi}
\quad\qquad\mbox{and}\qquad\quad
|A(x)\,\xi\cdot \bar{\zeta}|
\le
\Lambda\,|\xi|\,|\zeta|,
$$
for all $\xi,\zeta\in\mathbb{C}^n$ and almost every $x\in \mathbb{R}^n$. We have used the notation
$\xi\cdot\zeta=\xi_1\,\zeta_1+\cdots+\xi_n\,\zeta_n$ and therefore
$\xi\cdot\bar{\zeta}$ is the usual inner product in $\mathbb{C}^n$. Note
that then
$A(x)\,\xi\cdot\bar{\zeta}=\sum_{j,k}a_{j,k}(x)\,\xi_k\,\bar{\zeta_j}$.
Associated with this matrix we define the second order divergence
form operator
$$
L f
=
-{\rm div}(A\,\nabla f),
$$
which is understood in the standard weak sense as a maximal-accretive operator on the space $L^2(\mathbb{R}^n,dx)$ with domain $D(L)$ by means of a
sesquilinear form. Associated to this operator we can consider a functional calculus $\varphi(L)$ where $\varphi$ is holomorphic and bounded in some sector, the Riesz transform $\nabla L^{-1/2}$, and some square functions. The $L^p$ theory for these operators was developed in the monograph \cite{Aus}. The weighted norm inequalities were obtained in \cite{AM3} using a generalized Calder\'on-Zygmund theory from \cite{AM1}. As part of these results it was obtained that the commutators of these operators with $ \BMO$ functions also satisfy  weighted norm inequalities. It should be pointed out that the proof of the weighted norm inequalities for the associated commutators followed from the developed Calder\'on-Zygmund theory with a somehow technical adaptation of the proof of the weighted norm inequalities for the corresponding operators. As we are going to see next, a simple application of Corollary \ref{corol:restricted} gives the same estimates without any extra effort.
In order to apply Corollary \ref{corol:restricted} we just need to recall that the following weighted norm inequalities hold for $\varphi(L)$ and $\nabla L^{-1/2}$ (see \cite{AM3}):
$$
\|\varphi(L)f\|_{L^p(w)}
\le
C_w\,\|f\|_{L^p(w)},
\qquad
\forall\,  p_-<p<p_+,\ w \in A_{\frac{p}{p_-}} \cap RH_{(\frac{p_+}{p})'};
$$
and
$$
\|\nabla L^{-1/2}f\|_{L^p(w)}
\le
C_w\,\|f\|_{L^p(w)},
\qquad
\forall\,  q_-<p<q_+,\ w \in A_{\frac{p}{q_-}} \cap RH_{(\frac{q_+}{p})'};
$$
where $(p_,p_+)$ and $(q_-,q_+)$ are respectively the maximal intervals where the semigroup $\{e^{-t\,L}\}_{t>0}$ and its gradient $\{\sqrt{t}\,\nabla e^{-t\,L}\}_{t>0}$ are uniformly bounded on $L^p(\mathbb{R}^n)$. This clearly allows us to apply Corollary \ref{corol:restricted} and obtain the corresponding commutators results as desired. Similar results can be obtained for the square functions associated with $L$ (see \cite{AM3} for full details and references).

\subsection{Fractional operators associated with the Kato conjecture}
Very much as before we can consider the fractional operators $L^{-\alpha/2}$. The weighted norm inequalities for these were proved in \cite{AM-fract}. Using the same notation as before, if $p_-<p<q<p_+$ and $1/p - 1/q = \alpha / n$, then
$$
\|L^{-\alpha/2}f\|_{L^q (w^q)}
\leq
C_w
\|f\|_{L^p (w^p)},
\qquad
\forall\, w\in A_{1+\frac1{p_-}-\frac1p}\cap RH_{q\,\big(\frac{q_+}{q}\big)'}
.
$$
Note that Proposition \ref{prop:weights} (g) gives that the previous estimates hold for a class of weights that can be written as $w^\theta\in A_s$ for some $\theta>1$ and $s>1$. Then, we can apply Theorem \ref{thm:main:II} and conclude that, under the same hypotheses, for every $b\in \BMO$ and $k\ge 1$
$$
\|(L^{-\alpha/2})_b^kf\|_{L^q (w^q)}
\leq
C_w\,\|b\|_{ \BMO}^k
\|f\|_{L^p (w^p)},
\qquad
\forall\, w\in A_{1+\frac1{p_-}-\frac1p}\cap RH_{q\,\big(\frac{q_+}{q}\big)'}
.
$$
These estimates were proved in \cite{AM-fract} using a Calder\'on-Zygmund type result that allows one to extend the unweighted estimates for the fractional operators to the commutators. The present method immediately produces the same results once the weighted estimates have been obtained.

\bigskip

We continue our applications by considering the case of bilinear operators. Several of the statements can be generalized to $m$-linear operators, $m\geq 2$, an easy task that is left to the interested reader.

\subsection{Bilinear Calder\'on-Zygmund operators}
We say that a bilinear operator $T$ a priori defined from $\mathcal S\times\mathcal S$ into $\mathcal S'$ of the form
$$
T(f, g)(x)=\int_{\mathbb R^n}\int_{\mathbb R^n} K(x, y, z)f(y)g(z)\,dydz
$$
is a bilinear Calder\'on-Zygmund operator if it can be extended as a bounded operator from $L^{p_1}\times L^{p_2}$ to $L^p$ for all $1<p_1, p_2<\infty$ with $1/p_1+1/p_2=1/p$, and its distributional kernel $K$ coincides, away from the diagonal $\{(x,y,z) \in \mathbb R^{3n}: x=y=z \}$, with a function  $K(x,y,z)$ locally integrable which satisfies estimates of the form
$$|\partial^\alpha K(x,y,z)| \lesssim \big(|x-y| + |x-z| + |y-z|\big)^{-2n-|\alpha|}, |\alpha|\le 1.$$
The estimates on $K$ above are not the most general that one can impose in such theory, see \cite{GT}. An immediate consequence of Theorem \ref{thm:multilinear:main:II} and the known weighted boundedness of the bilinear Calder\'on-Zygmund operators \cite{Lerner2009} leads to the following result.  The bounds are sharp for the operator $T$ ($|\alpha|=0$), see \cite{LiMoenSun2014}.

\begin{theorem}
\label{thm:multilinear:CZ}
Let $T$ be a bilinear Calder\'on-Zygmund operator, $1<p_1,p_2,p<\infty$ be such that $\frac1p=\frac{1}{p_1}+\frac{1}{p_2}$, $\textbf{b}=(b_1, b_2)\in \BMO^2$, $\textbf{w}=(w_1, w_2)\in A_{\textbf P}$ and $\alpha$ a multi-index. Then, we have
$$\|[T, \textbf{b}]_\alpha \textbf{f}\|_{L^p (\nu_{\textbf{w}})} \lesssim \alpha ! [\textbf{w}]_{A_{\textbf P}}^{(|\alpha|+1)\max\{1, p_1'/p, p_2'/p\}}\prod_{j=1}^2 \|b_j\|^{\alpha_j}_{\BMO} \|f_j\|_{L^{p_j}\left(w_j\right)}.$$
\end{theorem}

The quantitative estimates in the above theorem are new. The limitations in the value of $p$ in our methods come from the use of Minkowski's inequality. However, for (non-quantitative) boundedness results the range of exponents can be extended so that $1/2<p<\infty$. See for example \cite{Lerner2009} for $|\alpha| =1$ and \cite{Perez2011} for $|\alpha| =2$. Alternatively, from Theorem \ref{thm:multilinear:CZ} the full range of exponents could be obtained by extrapolation (although not with optimal constants); see the recent work \cite[Corollary 1.5]{LMO}.

\subsection{Bilinear fractional integrals}
Consider now bilinear operators with positive kernels of the form
$$BI_s (f, g)(x)=\int_{\mathbb R^n}\int_{\mathbb R^n}\frac{1}{(|x-y|+|x-z|)^{2n-s}}f(y)g(z)\,dydz,$$
with $0<s<2n$. The appropriate class of vector weights to study these operators is defined as follows. For $1<p_1, p_2<\infty$ with $1/p_1+1/p_2>s/n$, let again $\textbf{P}=(p_1, p_2)$ and $q>0$ be such that
$$\frac{1}{q}=\frac{1}{p_1}+\frac{1}{p_2}-\frac{s}{n}.$$ The vector weight $\textbf{w}=(w_1, w_2)$ is said to belong to the class $A_{{\textbf P},q}$ if
$$[\textbf w]_{A_{{\textbf P},q}}:=\sup_{Q}\left(\fint_Q w_1^qw_2^q dx\right)\left(\fint_Q w_1^{-p_1'}dx\right)^{q/p_1'}\left(\fint_Q w_2^{-p_2'}dx\right)^{q/p_2'}<\infty.$$
It was shown by Moen \cite{Moen2009} and Chen and Xue \cite{CX2010} that $BI_\alpha: L^{p_1}(w_1^{p_1})\times L^{p_2}(w_2^{p_2})\to L^q(w_1^qw_2^q)$ as long as $\textbf w\in A_{{\textbf P},q}$. The dependence on the weighted constant was shown in \cite{LiMoenSun2015}.  Using this fact and the $A_{{\textbf P},q}$ version of our multilinear result, Theorem \ref{thm:multilinear:main:II}, we immediately recover the boundedness of the commutators of $BI_s$ with $\textbf{b}\in \BMO^2$ from the work in \cite{CX2010}. Specifically, for $\alpha $ a multi-index, we have
$$[BI_s, \textbf{b}]_{\alpha}: L^{p_1}(w_1^{p_1})\times L^{p_2}(w_2^{p_2})\to L^q(w_1^qw_2^q),$$
with appropriate quantitative estimates on the commutator operator norms similar to the ones stated in Corollary \ref{linear-fractional}.

\subsection{The bilinear Hilbert transform}
The bilinear Hilbert transform, defined via
$$BHT(f, g)(x)=\text{p.v}\,\int_{\mathbb R}f(x-t)g(x+t)\frac{dt}{t},$$
is a celebrity in harmonic analysis. It is a bilinear operator whose multiplier, unlike the ones for bilinear Calder\'on-Zygmund operators which are singular only at the origin, is singular along a line when viewed in the frequency plane. Let $1<p_1, p_2, p<\infty$ be such that $1/p_1+1/p_2=1/p$ and $\textbf{w}=(w_1, w_2)$ such that $w_1^{2p_1}\in A_{p_1}$ and $w_2^{2p_2}\in A_{p_2}$. Under these conditions, Culiuc, Di Plinio and Ou \cite{CPO2016} proved that
$$BHT: L^{p_1}(w_1^{p_1})\times L^{p_2}(w_2^{p_2})\to L^p(w_1^pw_2^p).$$
In view of Corollary \ref{cor:multilinear:II}, we immediately get the following new boundedness result for the commutator of $BHT.$

\begin{theorem}\label{thm:BHT:commutator}
Let $\textbf{b}\in \BMO^2$, $\textbf{w}=(w_1, w_2)$ be such that $w_j^2\in A_{p_j}$, where $1<p_j<\infty$, $j=1, 2$ and $1<p<\infty$ be so that $1/p=1/p_1+1/p_2$. Then, for any multi-index $\alpha$ we have
$$[BHT, \textbf{b}]_{\alpha}: L^{p_1}(w_1)\times L^{p_2}(w_2)\to L^p(\nu_{\textbf w}).$$
\end{theorem}

Combining this result with the extrapolation method from \cite{CM-extrapol} we can obtain weighted estimates for the commutators for  $\frac23<p\le 1$.

\begin{corollary}\label{corol:BHT:commutator}
Let $\textbf{b}\in \BMO^2$. Given $1<p_1,\,p_2<\infty$, let $p>2/3$ be so that
$\frac{1}{p}=\frac{1}{p_1}+\frac{1}{p_2}$.  Then, for any multi-index $\alpha$ we have
\begin{equation} \label{eqn:bht1-bis}
[BHT, \textbf{b}]_{\alpha}: L^{p_1}(w_1)\times L^{p_2}(w_2)\to L^p(\nu_{\textbf w}),
\end{equation}
for all $\textbf{w}=(w_1, w_2)$ such that
$w_j\in  A_{\max\{1, p_j/2\}}\cap RH_{\max\{1, 2/p_j\}}$.

\noindent In particular,
\begin{equation}
[BHT, \textbf{b}]_{\alpha}: L^{p_1}(|x|^{-a}) \times L^{p_2}(|x|^{-a})
\longrightarrow L^p(|x|^{-a}),
\label{eq:BH-power}
\end{equation}
if $a=0$ or if
\begin{equation}\label{eq:values-a}
1-\min\left\{
\max\left\{1,\frac{p_1}{2}\right\},\max\left\{1,\frac{p_2}{2}\right\}\right\}
<
a
<\min\left\{1, \frac{p_1}{2},\frac{p_2}{2}\right\}.
\end{equation}
As a result, \eqref{eq:BH-power} holds for all $a\in [0, 1/2)$.
\end{corollary}

We would like to emphasize that \eqref{eq:BH-power} with $a=0$ gives that $[BHT, \textbf{b}]_{\alpha}$ satisfies the same unweighted estimates as $BHT$.  The proof of Corollary \ref{corol:BHT:commutator} follows that of \cite[Corollary 1.23]{CM-extrapol}
by simply replacing $BHT$ by $[BHT, \textbf{b}]_{\alpha}$. The bottom line is that Theorem \ref{thm:BHT:commutator} provides the initial weighted norm inequalities needed to apply the extrapolation method and, in turn, these are the same for $BHT$ and $[BHT, \textbf{b}]_{\alpha}$. The same idea allows us to get a version of \cite[Theorem 1.18]{CM-extrapol} valid for $[BHT, \textbf{b}]_{\alpha}$, and also some vector-valued inequalities. Further details and the precise statements are left to the interested reader.

In Appendix \ref{appendix} we also state some additional weighted estimates for the commutators of the BHT with weights in some classes generalizing $A_{\textbf P}$, see Corollary \ref{corolo:BHT-new}.

\subsection{Multi-parameter operators}
An interesting application of our general result in a general setting is the boundedness of multi-parameter commutators of Calder\'on-Zygmund operators of product type. Moreover, we shall discover that if we consider the basis of rectangles $\mathcal R$, then the two flavors of spaces $BMO_{\mathcal R}$ and $\BMO_{\mathcal R}$ coincide.

For simplicity, we suppose that our underlying space is $\R^2=\R \times \R$, and we let $\mathcal R$ denote the family of rectangles with sides parallel to the axes. As explained in Subsection \ref{section:Comm-general}, we can consider the following two versions of spaces: $\BMO_{\mathcal R}$, defined via the norm
$$\|f\|_{\BMO_{\mathcal R}}
=\sup_{R \in \mathcal R} \|f-f_R\|_{{\rm exp}\,L,R},$$
and $BMO_{\mathcal R}$, defined via the norm
$$\|f\|_{BMO_{\mathcal R}}=\sup_{R \in \mathcal R} \fint_R |f-f_R| \, dx.$$
The space $BMO_{\mathcal R}$ is sometimes called ``little BMO"--since it is smaller than the $BMO(\R^2)$ space defined on cubes--and denoted $bmo$. Now, since
$$\|f\|_{BMO_{\mathcal R}}\leq \|f\|_{\BMO_{\mathcal R}},$$
we immediately obtain that $\BMO_{\mathcal R}\subset bmo$. It turns out that the converse inclusion
$$bmo\subset \BMO_{\mathcal R}$$
is also true, a ``little" fact previously unobserved to the best of our knowledge. In order to prove it, we need a few preliminary lemmas. In what follows, we will write $A_{2, \mathscr{B}}=A_{2, \mathscr{B}}(\R^2)$ to denote the class of $A_2$ weights defined with respect to the family $\mathscr{B}$, see again Proposition \ref{prop:weights:B}; and we will write $A_2=A_2(\R)$ for the usual $A_2$ class of weights, see Subsection \ref{subsection:Muckenhoupt}.

The lemma we first state below is proved in \cite[pp. 406-408]{RDF1985} for cubes, but the exact same proof actually holds for any family $\mathscr{B}$. Our interest will eventually be in applying these lemmas to the case where $\mathscr{B}=\mathcal R$.
\begin{lemma}
Let $f$ be a real valued locally integrable function. Then, $e^f$ is in $A_{2, \mathscr{B}}$ if and only if
$$
\sup_{B \in \mathscr{B}} \fint_Be^{|f-f_B|} \, dx =C_0<  \infty.
$$
Moreover, if $e^f \in A_2$, then $C_0 \leq 2 [e^f]_{A_{2, \mathscr{B}}}$.
\end{lemma}

\begin{lemma}
Let $f$ be a real valued locally integrable function.  If
$$
\sup_{B \in \mathscr{B}} \fint_B e^{|f-f_B|} \, dx = C_0<  \infty,
$$
then $f \in \BMO_{\mathscr{B}}$ and $\|f \|_{\BMO_{\mathscr{B}}} \leq \max \{1, \log_2 C_0\}$.
\end{lemma}

\begin{proof}
If $C_0\leq 2$, then clearly, by definition, $\|f\|_{\BMO_{\mathscr{B}}}\leq 1$. Assume thus that $C_0>2$ and let $\lambda\geq 1$. By Jensen's inequality, we have
$$
\fint_B e^{\frac{|f-f_B|}{\lambda}} \, dx \leq \left( \fint_B e^{{|f-f_B|}} \, dx\right)^{\frac{1}{\lambda}} = C_0^{\frac{1}{\lambda}} .
$$
Taking $\lambda = \log_2 C_0$ we obtain $\|f \|_{\BMO_{\mathcal \B}} \leq \log_2 C_0$.
\end{proof}
We point out that, if $f \in \BMO_{\mathscr{B}}$, then $\fint_B e^{|f-f_B| }\, dx$ may not be finite, but obviously
$$\fint_B e^{\frac{|f-f_B|}{\|f\|_{\BMO_{\mathscr{B}}}}} \, dx\leq 2.$$
We have the following immediate corollary of the previous two lemmas.

\begin{corollary}
Let $f$ be a real valued locally integrable function.  If $e^f \in A_{2, \mathscr{B}}$,
then $f \in \BMO_{\mathscr{B}}$ and
\begin{equation}\label{log2}
\|f \|_{\BMO_{\mathscr{B}}} \leq 1+\log_2 [e^f]_{A_{2, \mathscr{B}}}.
\end{equation}
\end{corollary}
Conversely, we recall that as observed above \eqref{eqn:BMO->Ap:3} in Lemma \ref{lemma:BMO->Ap} is valid for arbitrary bases provided $p>1$. Thus, given $f \in \BMO_{\mathscr{B}}$, for all $\lambda \geq \|f\|_{\BMO_{\mathscr{B}}}$, we then have
\begin{equation} \label{4}
\left[e^{\frac{f}{\lambda}}\right]_{A_{2, \mathscr{B}}} \leq 4^{\frac{ \|f\|_{\BMO_{\mathscr{B}}} }{\lambda}}.
\end{equation}

\begin{lemma}
A weight $w(x, y)$ is in $A_{2, \mathcal R}(\R^2)$ if and only if $w(x,\cdot)$ and $w(\cdot,y)$ are uniformly in $A_2(\R)$ for almost every $x, y \in \R$.
\end{lemma}
The proof of this result can be found in \cite[pp. 453-459]{RDF1985}. The arguments there show that 
\begin{equation}\label{weightsrelationship:1}
\essup_x[w(x,\cdot)]_{A_2}\le [w]_{A_{2, \mathcal R}},
\qquad
\essup_y[w(\cdot,y)]_{A_2}  \le [w]_{A_{2, \mathcal R}},
\end{equation}
and
\begin{equation}\label{weightsrelationship:2}
[w]_{A_{2, \mathcal R}}  \le C_n\,\essup_x[w(x,\cdot)]_{A_2}\, \essup_y[w(\cdot,y)]_{A_2},
\end{equation}
where $C_n$ is a dimensional constant.

A similar statement holds for $bmo$ functions. The following result was proved in \cite[pp. 279-281]{CS1996}.
\begin{lemma}
Let $f$ be a locally integrable function on $\R^2$.  Then, $f$ is in $bmo$ if and only if
$f(x, \cdot)$ and $f(\cdot, y)$ are uniformly in $BMO(\R)$ for almost every $x, y \in \R$.
Moreover,
\begin{equation}\label{bmorelationship}
\|f(x, \cdot)\|_{BMO} + \|f(\cdot, y)\|_{BMO} \approx \|f\|_{bmo}.
\end{equation}
\end{lemma}
Finally, we are able to prove the following
\begin{theorem}\label{thm:bmo-BMO}
Let $f$ be a locally integrable function on $\R^2$.  Then, $f\in bmo$ if and only if $f\in \BMO_{\mathcal R}$.
\end{theorem}

\begin{proof}
It suffices to show that $\|f\|_{\BMO_{\mathcal R}}\lesssim \|f\|_{bmo}$. Let then $f\in bmo$. By \eqref{bmorelationship}
$$
\|f(x, \cdot)\|_{BMO(\R)} + \|f(\cdot, y)\|_{BMO(\R)} \lesssim \|f\|_{bmo(\R^2)}.
$$
By the John-Nirenberg inequality in $\R$ we have that, for some {\it fixed} $c>0$,
$$
\|f(x, \cdot)\|_{\BMO(\R)} + \|f(\cdot, y)\|_{\BMO(R)} \leq c \|f\|_{bmo(\R^2)}.
$$
By \eqref{4} applied in $\R$,
$$
\left[e^{\frac{f(x,\cdot)}{c \|f\|_{bmo}}}\right]_{A_2(\R)}, \,\,\,\left[e^{\frac{f(\cdot,y)}{c \|f\|_{bmo}}}\right]_{A_2(\R)} \leq 4,
$$
which combined with \eqref{weightsrelationship:2} gives
$$
\left[e^{\frac{f}{c \|f\|_{bmo}}}\right]_{A_{2, \mathcal R}} \leq  16 C_n.
$$
Now, by \eqref{log2} ,  $\frac{f}{c \|f\|_{bmo}} \in \BMO_{\mathcal R}$ and
$$
\Big|\Big|\frac{f}{c \|f\|_{bmo}}\Big|\Big|_{\BMO_{\mathcal R}} \leq  5+\log_2 C_n,
$$
and finally
$$
\|f\|_{\BMO_{\mathcal R}} \leq  (5+\log_2 C_n)\,c\, \|f\|_{bmo}.
$$
\end{proof}

The following result was obtained by Ferguson and Sadosky \cite[Theorem 2.1]{FS2000}  on the 2-dimen\-sio\-nal torus.  The version we state follows from Theorems \ref{thm:bmo-BMO} and \ref{thm:main:B}. The needed weighted estimates to use our approach in the 2-dimensional torus were obtained by Cotlar and Sadosky \cite{CS1996} and in the Euclidean setting are due to Fefferman and Stein \cite{FS1982}.
\begin{theorem}
Let $H_j$ denote the 1-dimensional Hilbert transforms in the $j$-th variable, $j=1, 2$ and let $b\in bmo$. Then, for all $f\in L^p(\R^2), 1<p<\infty$, we have
$$\|[b, H_1H_2]\|_{L^p(\R^2)}\lesssim \|b\|_{bmo}\|f\|_{L^p(\R^2)}.$$
\end{theorem}

The results above about $bmo$ could be easily generalized to other product settings $\R^n \times \R^m$ and if $T$ is Calder\'on-Zygmund operator of product type in the sense of \cite{FS1982} then we obtain with similar arguments the following.
\begin{theorem}
Let $T$ be a Calder\'on-Zygmund operator of product type and let $b$ be a function in  $bmo(\R^n \times \R^m)$.  Then, for  all $f\in L^p(\R^n \times \R^m), 1<p<\infty$, we have
$$\|[b, T]\|_{L^p(\R^n \times \R^m)}\lesssim \|b\|_{bmo}\|f\|_{L^p(\R^n \times \R^m)}.$$
\end{theorem}

\subsection{Operators on spaces of homogeneous type}
All the previous applications, as long as the operators can be defined, can be extended to the case where the underlying Euclidean space is replaced by a space of homogeneous type. The details are left to the interested reader.

\appendix

\section{Multilinear commutators III}\label{appendix}

Shortly after this paper was posted on arXiv, a new extrapolation result was obtained in \cite{LMO} associated to the classes of weights $A_{\textbf{P},\textbf{R}}$ which generalize the classes $A_{\textbf{P}}$ introduced above. As a matter of fact, \cite[Section 2.5]{LMO} borrowed some of the key ideas from the present paper and sketched an argument that yields a commutator result along the lines of Theorem \ref{thm:multilinear:main:II}. We present here the complete argument with the quantitative bounds.

We begin by introducing the class of weights $A_{\textbf{P},\textbf{R}}$. Let $\textbf{R}=(r_1,\dots, r_{m+1})\in [1,\infty)^{m+1}$,   and $\textbf{P} = (p_1, \dots, p_m)$ with
\begin{equation}\label{R-P}
r_j<p_j<\infty,\quad j=1,\dots,m;
\qquad
\text{and} \qquad r_{m+1}'>p,
\quad \text{where}\quad
\frac{1}{p}=\frac{1}{p_1} + \cdots + \frac{1}{p_m}.
\end{equation}
Let $\textbf{w} = (w_1, \dots, w_m)$ and set
$$
\nu_{\textbf{w}} = \prod_{j=1}^m w_j^{p/p_j}.
$$
We say that the vector weight $\textbf{w} \in A_{\textbf{P},\textbf{R}}$
if
$$
[\textbf{w}]_{A_{\textbf{P},\textbf{R}}}
:=
\sup_Q \left( \fint_Q \nu_w^{{\Delta_{m+1}}/{p}} \,dx \right)^{1/\Delta_{m_+1}} \prod_{j=1}^m \left( \fint_Q w_j^{-{\Delta_j}/{p_j}} \,dx \right)^{1/\Delta_j} < \infty,$$
where
$$
\frac1{\Delta_j}=\frac1{r_j}-\frac1{p_j},\quad j=1,\dots,m;
\qquad
\text{and} \qquad
\frac1{\Delta_{m+1}}
=\frac1{p}-\frac1{r_{m+1}'}.
$$
It is straightforward to see that $A_{\textbf{P},\textbf{R}}=A_{\textbf{P}}$ whenever $\textbf{R}=(1,\dots,1).$ See also \cite{CTW} for other similar classes of weights.

We next state the aforementioned result (the present formulation differs from \cite[Theorem 2.21]{LMO} where the emphasis was put on extrapolation):
\begin{theorem}\label{thm:multilinear:main:III}
	Let $T$ be an $m$-linear operator. Let $\textbf{R}=(r_1,\dots, r_{m+1})\in [1,\infty)^{m+1}$,   and $\textbf{P} = (p_1, \dots, p_m)$ so that \eqref{R-P} holds and assume further that $p>1$.
	Suppose that there exists an increasing functions $\phi: [1, \infty)\to [0, \infty)$ such that for all $\textbf{w} = (w_1, \dots, w_m)\in A_{\textbf{P},\textbf{R}}$, we have
	\begin{equation}
	\label{vector-weight:III}
	\|T\textbf{f}\|_{L^p (\nu_{\textbf{w}})}
	\lesssim
	\phi \left([\textbf{w}]_{A_{\textbf{P},\textbf{R}}}\right)\prod_{j=1}^m \|f_j\|_{L^{p_j}\left(w_j\right)}.
	\end{equation}
	Then, for all $\textbf{b} = (b_1, \dots, b_m) \in  \BMO^m$ and for each multi-index $\alpha$, we have
	\begin{equation}
	\label{multi-commutator-III}
	\|[T, \textbf{b}]_\alpha \textbf{f}\|_{L^p (\nu_{\textbf{w}})}
	\lesssim
	\alpha !\,\phi\left(c_{\textbf{P},\textbf{R}} [\textbf{w}]_{A_{\textbf{P},\textbf{R}}}\right)
	 [\textbf{w}]_{A_{\textbf{P},\textbf{R}}}^{|\alpha|\max\{\Delta_1,\ldots,\Delta_{m+1}\}}
	\prod_{j=1}^m \|b_j\|^{\alpha_j}_{\BMO} \|f_j\|_{L^{p_j}\left(w_j\right)},
	\end{equation}
	where
	$c_{\textbf{P},\textbf{R}}
	=
	2^{\frac{r}{1-r}+2\sum_{j=1}^m \min\{p_j/\Delta_j, p/\Delta_{m+1}\}/p_j}$.
\end{theorem}

\begin{proof}
The proof follows the ideas of Theorem \ref{thm:multilinear:main:II} and we only consider the bilinear case.  Without loss of generality, we assume that $b_1, b_2$ are real valued and normalized so that their $\BMO$ norms are equal to $1$. As before, we are going to use the Cauchy integral trick, and given $\textbf{w}\in A_{\textbf{P},\textbf{R}}$ everything reduces to showing that for some appropriate $\delta_1,\delta_2>0$ (to be chosen later) and for $|z_1|=\delta_1$, $|z_2|=\delta_2$,
we have
$$
\widetilde{\textbf{w}}:=
(\widetilde{w}_1,\widetilde{w}_2)
:=
(w_1 e_{b_1}, w_2 e_{b_2})
:=
(w_1 e^{-\text{Re}(z_1)p_1 b_1} , w_2 e^{-\text{Re}(z_2)p_2 b_2}) \in A_{\textbf{P},\textbf{R}}.
$$
By \cite[Lemma 5.3(i)]{LMO} it follows that $\nu_{\textbf{w}}^{{\Delta_{3}}/{p}} \in A_{\frac{1-r}{r}\Delta_{3}}$ and $w_j^{{\theta_j }/{p_j}}\in A_{\frac{1-r}{r}\theta_j }$ or, equivalently, $w_j^{-{\Delta_j}/{p_j}}\in A_{\frac{1-r}{r}\Delta_j}$ for $j=1,2$,
where
$$
\frac1r=\frac1{r_1}+\frac1{r_3}+\frac1{r_3}>\frac1{p_1}+\frac1{p_2}+1-\frac1p=1;
\qquad\qquad
\frac1{\theta_j}=\frac{1-r}{r}-\frac1{\Delta_j},\quad j=1,2.
$$
Moreover,
$$
\Big[\nu_{\textbf{w}}^{{\Delta_{3}}/{p}}\Big]_{A_{\frac{1-r}{r}\Delta_{3}}}
\le
 [\textbf{w}]_{A_{\textbf{P},\textbf{R}}}^{\Delta_{3}};
 \qquad\quad
\Big[w_j^{-{\Delta_j}/{p_j}}\Big]_{ A_{\frac{1-r}{r}\Delta_j }}
=
\Big[w_j^{{\theta_j }/{p_j}}\Big]_{ A_{\frac{1-r}{r}\theta_j}}^{\frac{1-r}{r}\Delta_j-1}
\le
 [\textbf{w}]_{A_{\textbf{P},\textbf{R}}}^{\Delta_{j}},\quad j=1,2.
$$
Using now Lemma \ref{RH-sharp} and writing, for a given weight $w$, $\rho(w)$ instead of $\rho_w$, we can find $\varrho=\varrho(\bw)=\min \{\rho (\nu_{\textbf{w}}^{{\Delta_{3}}/{p}}), \rho (w_1^{-{\Delta_1}/{p_1}}),\rho (w_2^{-{\Delta_2}/{p_2})} \}>1$ so that
\begin{equation}\label{eq:afrfr****}
\varrho'\sim \max
\Big\{
\Big[\nu_{\textbf{w}}^{{\Delta_{3}}/{p}}\Big]_{A_{\frac{1-r}{r}\Delta_{3}}},
\Big[w_1^{-{\Delta_1}/{p_1}}\Big]_{ A_{\frac{1-r}{r}\Delta_1 }},\Big[w_2^{-{\Delta_2 }/{p_2}}\Big]_{ A_{\frac{1-r}{r}\Delta_2 }}\Big\}
\le
[\textbf{w}]_{A_{\textbf{P},\textbf{R}}}^{\max\{\Delta_1,\Delta_2,\Delta_3\}}
\end{equation}
and the following reverse H\"older inequalities hold:
\begin{equation}\label{rh1:**}
\left( \fint_Q \nu_{\textbf{w}}^{{\Delta_{3}}\varrho/{p}} dx \right)^{1/\varrho} \leq 2 \fint_Q \nu_{\textbf{w}}^{{\Delta_{3}}/{p}} dx
\end{equation}
and, for $j=1,2$,
\begin{equation}\label{rh2:***}
\left( \fint_Q w_j^{-{\Delta_j}\varrho/{p_j}}\,dx \right)^{1/\varrho} \leq 2 \fint_Q w_j^{-{\Delta_j}/{p_j}}\,dx.
\end{equation}
Using all these and regrouping terms,  we get
\begin{align*}
&
\Big(\fint_Q
\nu_{\widetilde{\bw}}^{{\Delta_{3}}/{p}} d x\Big)^{1/{\Delta_{3}}}
\prod_{j=1}^2 \Big(\fint_Q \widetilde{w}_j^{-{\Delta_i}/{p_i}} d x\Big)^{1/{\delta_j}}
\\
&\qquad\qquad=
\left( \fint_Q \nu_{\bw}^{{\Delta_{3}}/{p}} e_{b_1}^{ {\Delta_{3}}/{p_1}} e_{b_1}^{ {\Delta_{3}}/{p_2}} dx \right)^{1/{\Delta_{3}}}
\prod_{j=1}^2 \Big(\fint_Q w_j^{-{\Delta_j}/{p_j}} e_{b_j}^{-{\Delta_j}/{s_j}} d x\Big)^{1/{\Delta_j}}
\\
&\qquad\qquad\le
\left( \fint_Q \nu_{\bw}^{{\Delta_{3}}\varrho/p} dx \right)^{1/({\Delta_{3}\varrho})}
\prod_{j=1}^2 \Big(\fint_Q w_j^{-{\Delta_j}\varrho/{p_j}} d x\Big)^{1/({\Delta_j\varrho})}
\\
&\qquad\qquad\quad\qquad
\left( \fint_Q \prod_{j=1}^2 e_{b_j}^{ {\Delta_{3}}\varrho'/{p_j}} dx \right)^{1/({\Delta_{3}\varrho'})}
\prod_{j=1}^2 \Big(\fint_Q e_{b_j}^{-{\Delta_j}\varrho'/{p_j}} d x\Big)^{1/({\Delta_j\varrho'})}
\\
&
\qquad\qquad\le
2^{\frac{1-r}{r}}[\textbf{w}]_{A_{\textbf{P},\textbf{R}}}
\prod_{j=1}^2
\left( \fint_Q e_{b_j}^{{\Delta_{3}}\varrho'/{p}} dx \right)^{ p/({\Delta_{3}\varrho'p_j})}
\Big(\fint_Q e_{b_j}^{-{\Delta_j}\varrho'/{p_j}} d x\Big)^{1/({\Delta_j\varrho'})}
\\
&
\qquad\qquad\le
2^{\frac{1-r}{r}}[\textbf{w}]_{A_{\textbf{P},\textbf{R}}}
\prod_{j=1}^2 \Big[e_{b_j}^{ {\Delta_{3}}\varrho'/p}\Big]_{A_{1+{\Delta_{3}p_j}/({\Delta_j p})}}^{ p/({\Delta_{3}\varrho'p_j})}
\\
&
\qquad\qquad\le
2^{\frac{1-r}{r}+2(\delta_1+\delta_2)} [\textbf{w}]_{A_{\textbf{P},\textbf{R}}},
\end{align*}
where the last estimate holds by Lemma \ref{lemma:BMO->Ap} provided
$$
\delta_j\le \frac1{\varrho'p_j}\min \Big\{\frac{p_j}{\Delta_j},\frac{p}{\Delta_{3}}\Big\}.
$$
Notice that this choice implies that $\delta_j\le \min\{p_j/\Delta_j, p/\Delta_3\}/p_j$. On the other hand, recalling \eqref{eq:afrfr****} and assuming further that
$\delta_j\sim [\textbf{w}]_{A_{\textbf{P},\textbf{R}}}^{\max\{\Delta_1,\Delta_2,\Delta_3\}}
$ we eventually obtain
$$
\|[T, \textbf{b}]_\alpha \textbf{f}\|_{L^p (\nu_{\textbf{w}})} \lesssim \alpha !\delta_1^{-\alpha_1}\delta_2^{-\alpha_2}
\phi\left(c_{\textbf{P},\textbf{R}}[\textbf{w}]_{A_{\textbf{P},\textbf{R}}}\right)  \|f_1\|_{L^{p_1}\left(w_1\right)}\|f_2\|_{L^{p_2}\left(w_2\right)}.
$$
with $c_{\textbf{P},\textbf{R}}
=
2^{\frac{r}{1-r}+2\min\{p_1/\Delta_1, p/\Delta_3\}/p_1+2\min\{p_2/\Delta_2, p/\Delta_3\}/p_2}$. This easily gives the desired estimate.
\end{proof}

Using Theorem \ref{thm:multilinear:main:III} and \cite[Theorem 3]{CPO2016} we can immediately obtain the following weighted estimates for the commutators of the BHT (note that, as before, one can remove the restriction $p>1$ by using extrapolation, see \cite[Corollary 2.26]{LMO}):
\begin{corollary}\label{corolo:BHT-new}
Let $\textbf{R}=(r_1,r_2,r_3)$ be such that $1<r_1,r_2,r_3<\infty$ and
\begin{equation}\label{cond-adm:corol}
\frac1{\min\{r_1,2\}}+\frac1{\min\{r_2,2\}}+\frac1{\min\{r_3,2\}}<2.
\end{equation}
Let $\textbf{P}=(p_1,p_2)$ be such that  $r_1<p_1<\infty$, $r_2<p_2<\infty$, and $1<p<r_3'$ where $1/p=1/p_1+1/p_2$. For every $\textbf{b}\in \BMO^2$, $\textbf{w}=(w_1, w_2)\in A_{\textbf{P},\textbf{R}}$ and any multi-index $\alpha$ we have
		$$[BHT, \textbf{b}]_{\alpha}: L^{p_1}(w_1)\times L^{p_2}(w_2)\to L^p(\nu_{\textbf w}).$$	
\end{corollary}

\end{document}